\documentclass[11pt,english]{article}
\usepackage{lmodern}

\usepackage[T1]{fontenc}
\usepackage[latin9]{inputenc}
\usepackage{geometry}
\usepackage{xcolor}
\usepackage{babel}
\usepackage{pdfcolmk}
\usepackage{amsmath}
\usepackage{amsthm}
\usepackage{amssymb}
\usepackage{esint}
\PassOptionsToPackage{normalem}{ulem}
\usepackage{ulem}
\usepackage[unicode=true,pdfusetitle,
 bookmarks=true,bookmarksnumbered=false,bookmarksopen=false,
 breaklinks=false,pdfborder={0 0 1},backref=false,colorlinks=false]
 {hyperref}
\usepackage{breakurl}

\makeatletter

\providecolor{lyxadded}{rgb}{0,0,1}
\providecolor{lyxdeleted}{rgb}{1,0,0}

\numberwithin{equation}{section}
\numberwithin{figure}{section}
\numberwithin{table}{section}
\usepackage{enumitem}		
\newcommand{\lyxaddress}[1]{
\par {\raggedright #1
\vspace{1.4em}
\noindent\par}
}
\theoremstyle{plain}
\newtheorem{thm}{\protect\theoremname}
  \theoremstyle{plain}
  \newtheorem{lem}[thm]{\protect\lemmaname}

\newcommand {\norm} [1] { \lVert #1 \rVert}

\newcommand {\abs} [1] {\left| #1 \right|}

\makeatother

\usepackage{babel}
  \providecommand{\lemmaname}{Lemma}
\providecommand{\theoremname}{Theorem}

\begin{document}

\title{Prevalence of Delay Embeddings with a Fixed Observation Function}

\date{$\:$}

\author{Raymundo Navarrete and Divakar Viswanath }

\maketitle

\lyxaddress{Department of Mathematics, University of Michigan (raymundo/divakar@umich.edu). }
\begin{abstract}
Let $x_{j+1}=\phi(x_{j})$, $x_{j}\in\mathbb{R}^{d}$, be a dynamical
system with $\phi$ being a diffeomorphism. Although the state vector
$x_{j}$ is often unobservable, the dynamics can be recovered from
the delay vector $\left(o(x_{1}),\ldots,o(x_{D})\right)$, where $o$
is the scalar-valued observation function and $D$ is the embedding
dimension. The delay map is an embedding for generic $o$, and more
strongly, the embedding property is prevalent. We consider the situation
where the observation function is fixed at $o=\pi_{1}$, with $\pi_{1}$
being the projection to the first coordinate. However, we allow polynomial
perturbations to be applied directly to the diffeomorphism $\phi$,
thus mimicking the way dynamical systems are parametrized. We prove
that the delay map is an embedding with probability one with respect
to the perturbations. Our proof introduces a new technique for proving
prevalence using the concept of Lebesgue points.
\end{abstract}
\global\long\def\norm#1{\left|\left|#1\right|\right|}

\global\long\def\set#1{\left\{  #1\right\}  }

\global\long\def\abs#1{\left|#1\right|}

\global\long\def\falling#1#2{#1^{\underbar{\relsize{-2}\ensuremath{#2}}}}

\global\long\def\e{\mathbb{E}}

\global\long\def\p{\mathbb{P}}

\global\long\def\strr#1#2{\genfrac{\{}{\}}{0pt}{}{#1}{#2} }

\global\long\def\var{\mathrm{Var}}

\global\long\def\floor#1{\lfloor#1\rfloor}

\global\long\def\ceil#1{\lceil#1\rceil}

\section{Introduction}

Let $x_{j+1}=\phi(x_{j})$ be a dynamical system. If $o$ is a scalar
valued observation function, the delay map is given by 
\[
F_{0}(x_{1})=\left(o(x_{1}),\ldots,o(x_{D})\right).
\]
The question of when $F_{0}$ is an embedding was considered by Aeyels
\cite{Aeyels1980} and Takens \cite{Takens1981}. Suppose that $x_{j}\in\mathbb{R}^{n}$
but with the dynamics confined to an invariant submanifold of dimension
$d\leq n$. Alternatively, we may assume $x_{j}\in\mathfrak{m}$,
where $\mathfrak{m}$ is a manifold of dimension $d$. Based on an
analogy to Whitney embedding \cite{Hirsch1976}, we may expect $F_{0}$
to be an embedding for generic $o$ for embedding dimension $D\geq2d+1$.
Here genericity is with respect to the space of functions $o$ under
a $C^{r}$ topology with $r\geq2$ \cite{Hirsch1976}.

Sauer et al \cite{SauerYorkeCasdagli1991} introduced a new point
of view, supported by deep ideas, into the theory of delay embeddings.
If $x\in\mathbb{R}^{n}$ and $\alpha\in\mathbb{Z}_{0}^{n}$ is a multi-index,
denote the monomial $x^{\alpha}$ by $p_{\alpha}(x)$. Instead of
assuming the observation function $o$ to be any $C^{r}$ function,
Sauer et al take the observation function to be the sum of some fixed
function $o^{\ast}$ and a finite linear combination of the monomials
$p_{\alpha}(x)$. Proofs of genericity rely on ``bump'' functions
or $C^{\infty}$ functions with compact support. Although the device
of bump functions is of much utility in differential topology \cite{Hirsch1976},
bump functions hardly ever arise in applications. In contrast, physical
models often use polynomials. Thus, limiting the perturbations to
a finite linear combination of polynomials is a welcome shift in point
of view. 

A property is generic in a Baire space if it holds for a countable
intersection of open and dense sets. A generic set is always dense
but it may be of probability zero (in a reasonable sense). For example,
generic subsets of $[0,1]$ of probability zero may be constructed
easily. Thus, it may be questioned whether the concept of genericity
captures the notion of what is typical in applications.

Sauer et al \cite{SauerYorkeCasdagli1991} answered that question
by introducing the notion of prevalence. To say that delay embeddings
are prevalent is equivalent to saying that the delay map is an embedding
for almost every linear combination of polynomials. If probabilities
are defined by normalizing the Lebesgue measure, we may say that the
delay map is an embedding with probability one.

Suppose $x_{k}=\phi^{k}(x_{1})$ and $y_{k}=\phi^{k}(y_{1})$. For
$F_{0}$ to be an injection, we must have $F_{0}(x_{1})\neq F_{0}(y_{1})$
whenever $x_{1}\neq y_{1}$. A major difficulty in the proof of injectivity
arises in handling points $x_{1}\neq y_{1}$ but with overlapping
orbits. For example, we may have $y_{1}=x_{2}$ or $y_{1}=x_{3}$.
Related difficulties arise in handling periodic points and in the
proof of immersivity (an embedding must be injective as well as immersive).
Sauer et al \cite{SauerYorkeCasdagli1991} introduced several key
ideas for handling these difficulties. However, there is a minor gap
in their proof. In section 4, we fix that gap and show that earlier
mathematical treatments have serious deficiencies. Therefore, proofs
prior to Sauer et al cannot be accepted.

The proof of Sauer et al \cite{SauerYorkeCasdagli1991} is quite informal.
We give a more formally precise development of their ideas in sections
2 and 3. Later, we consider the case where the observation map is
fixed at $o=\pi_{1}$, with $\pi_{1}$ being the projection to the
first coordinate and with polynomial perturbations applied directly
to $\phi$. Ideas essential for the new developments are interspersed
in sections 2 and 3. Sauer et al include a filtering step applied
to the delay map in their main theorems. In addition to mathematical
informality, the filtering step makes the essential ideas difficult
to grasp and verify. Thus, the filtering step is omitted in section
4, where we derive their main results in a modified form.

From section 5 onwards, we treat the case where $o=\pi_{1}$ and $\phi$
itself is perturbed by polynomials. There are two main motivations
for considering this case. First, from a purely aesthetic point of
view, it is desirable to make the theory of delay embeddings depend
upon the dynamics and not the observation function. Second, the setting
with $o=\pi_{1}$ is pertinent to applications. For example, the most
natural way to extract a time series from a fluid flow is to simply
record the fluid velocity at a fixed point \cite{Takens1981}.

The main technical novelty in our approach is related to the concept
of Lebesgue points. Our delay embedding theorem for the $o=\pi_{1}$
case requires $D\geq4d+2$, although our earlier work \cite{NavarreteViswanath17}
suggests $D\geq2d+1$. In the concluding section, we express the hope
that the technique of Lebesgue points may prove useful in obtaining
prevalence versions of some classical results in dynamical systems
theory. In that regard, we mention the extensions of delay embedding
theory to PDE by Robinson \cite{Robinson2005,Robinson2011}. A more
complete account of other mathematical investigations in embedding
theory may be found in the introduction to our earlier work \cite{NavarreteViswanath17}.

\section{Transfer of volume}

A key idea in the work of Sauer et al \cite{SauerYorkeCasdagli1991}
is to transfer volumes from embedding space to parameter space. For
an example of what we mean by transfer of volume, suppose $A$ is
a square matrix. Then a volume equal to $\mathfrak{v}$ in the range
is transferred to $\mathfrak{v}/\det A$ in the domain.

Suppose $G:\mathbb{R}^{D_{\alpha}}\times\mathbb{R}^{\mathfrak{d}}\rightarrow\mathbb{R}^{D}$
is a $C^{r}$ function with $r\geq2$. Here $\mathbb{R}^{D_{\alpha}}$
is the space of parameters and we will denote a point in parameter
space by $\left(c_{\alpha}\right)$ or $c_{\alpha}$, with the understanding
that $\left(c_{\alpha}\right)$ (or $c_{\alpha}$)is a column vector.
The transfer of volume is carried out with fixed $\mathfrak{z}\in\mathbb{R}^{\mathfrak{d}}$.
Thus, the dependence of $G\left(c_{\alpha},\mathfrak{z}\right)$ on
$\mathfrak{z}$, which will be nonlinear, does not come up in the
transfer of volume argument. When the map $\phi$ is fixed and only
the observation function is parametrized, $G$ is linear in the parameters
$c_{\alpha}$. The embedding space is $\mathbb{R}^{D}$ and the dimension
$D$ of this space is of much importance. The rank of $G$ is mainly
constrained by $D$ because $D_{\alpha}\gg D$, and the rank determines
how much volume (or how little, with lesser the better) is transferred
from embedding space to parameter space.

In the following lemma and later we refer to $\mu(B_{1}\cap B_{2})/\mu(B_{2})$,
where $\mu(\cdot)$ is the Lebesgue measure, as the probability of
$B_{1}$ relative to $B_{2}$ (both sets are assumed to be measurable).
Measure will always refer to Lebesgue measure. The following lemma
transfers the volume of a ball of radius $L\epsilon$ in $\mathbb{R}^{D}$
to parameter space. All norms in this paper are spectral or $L^{2}$
norms.
\begin{lem}[\cite{SauerYorkeCasdagli1991}]
 Let $\mathfrak{g}\left(c_{\alpha}\right)=A\left(c_{\alpha}\right)+\mathfrak{g}_{0}$
be a linear (affine) map from $\mathbb{R}^{D_{\alpha}}$ to $\mathbb{R}^{D}$,
with $A$ being a $D\times D_{\alpha}$ matrix. Suppose the first
$\mathfrak{r}$ singular values of $A$ are at least as great as $\sigma>0$.
Then the measure of the set 
\begin{equation}
\set{c_{\alpha}\Bigl|\norm{A(c_{\alpha})+\mathfrak{g}_{0}}\leq L\epsilon}\cap\set{c_{\alpha}\Bigl|\norm{c_{\alpha}}\leq a}\label{eq:lem1-1}
\end{equation}
is less than or equal to 
\begin{equation}
2^{D_{\alpha}}L^{\mathfrak{r}}\epsilon^{\mathfrak{r}}a^{D_{\alpha}-\mathfrak{r}}\Bigl/\sigma^{\mathfrak{r}},\label{eq:lem1-2}
\end{equation}
and the probability of $\norm{A\left(c_{\alpha}\right)+\mathfrak{g}_{0}}\leq L\epsilon$
relative to $\norm{c_{\alpha}}\leq a$ is less than or equal to 
\[
D_{\alpha}!L^{\mathfrak{r}}\epsilon^{\mathfrak{r}}\bigl/\sigma^{\mathfrak{r}}a^{\mathfrak{r}}.
\]
\label{lem:1-transfer-of-volume}\end{lem}
\begin{proof}
Suppose $\mathfrak{u}_{1},\ldots,\mathfrak{u}_{D_{\alpha}}$are the
right singular vectors, $\mathfrak{v}_{1},\ldots,\mathfrak{v}_{D}$
are the left singular vectors, and $\sigma_{1},\ldots,\sigma_{D_{\alpha}}$
the singular values of $A$. (see \cite{TrefethenBau1997}). Let $\left(c_{\alpha}\right)=\sum_{i=1}^{D_{\alpha}}\mathfrak{c}_{i}\mathfrak{u}_{i}$
and $\mathfrak{g_{0}}=\sum_{i=1}^{D}\mathfrak{g}_{i}\mathfrak{v}_{i}$.

For $i=1,\ldots,\mathfrak{r}$, $\norm{A(c_{\alpha})+\mathfrak{g}_{0}}\leq L\epsilon$
implies that $\abs{\sigma_{i}\mathcal{\mathfrak{c}}_{i}+\mathfrak{g}_{i}}\leq L\epsilon$
and therefore $\abs{\mathfrak{c}_{i}+\mathfrak{g}_{i}/\sigma_{i}}\leq L\epsilon/\sigma_{i}\leq L\epsilon/\sigma$.
Thus, the coefficient $\mathfrak{c}_{i}$ must lie in an interval
of measure less than $2L\epsilon/\sigma$ for $i=1,\ldots,\mathfrak{r}$.

For $i=\mathfrak{r}+1,\ldots,D_{\alpha}$, $\norm{c_{\alpha}}\leq a$
implies that $\mathfrak{c}_{i}$ must vary inside the interval $[-a,a]$,
whose length is $2a$.

Therefore, the volume of the set (\ref{eq:lem1-1}) is bounded above
by $(2L\epsilon/\sigma)^{\mathfrak{r}}(2a)^{D_{\alpha}-\mathfrak{r}}$,
which simplifies to (\ref{eq:lem1-2}).

For the statement about the probability of $\norm{A\left(c_{\alpha}\right)+\mathfrak{g}_{0}}\leq L\epsilon$
relative to $\norm{c_{\alpha}}\leq a$, we divide (\ref{eq:lem1-2})
by $\gamma a^{D_{\alpha}}$, where $\gamma$ is the volume of the
unit sphere in $\mathbb{R}^{D_{\alpha}}$, to obtain 
\[
2^{D_{\alpha}}L^{\mathfrak{r}}\epsilon^{\mathfrak{r}}\Bigl/\gamma\sigma^{\mathfrak{r}}a^{\mathfrak{r}}.
\]
The proof is completed using $\gamma=\pi^{D_{\alpha}/2}\bigl/\Gamma\left(D_{\alpha}/2+1\right)\geq2^{D_{\alpha}}\bigl/D_{\alpha}!$.
\end{proof}
Lemma \ref{lem:1-transfer-of-volume} shows how a volume $\norm{\mathfrak{g}(c_{\alpha})}\leq L\epsilon$
in embedding space is transferred to a probability relative to $\norm{c_{\alpha}}\leq a$
in parameter space. The transferred probability is proportional to
$\epsilon^{\mathfrak{r}}$, and therefore, as the rank $\mathfrak{r}$
increases, the probability becomes smaller.

To obtain prevalence with the observation function fixed and the map
parametrized, we will rely on the following nonlinear transfer of
volume lemma. When the previous Lemma \ref{lem:1-transfer-of-volume}
is applied, $L$ will be a Lipshitz constant. When the following lemma
is applied, $L$ will be a Lipshitz constant as well as a bound on
the quadratic remainder term in a Taylor series.
\begin{lem}
Suppose $\mathfrak{g}:\mathbb{R}^{D_{\alpha}}\rightarrow\mathbb{R}^{D}$
is a $C^{2}$ function, with the Taylor series $\mathfrak{g}\left(c_{\alpha}\right)=\mathfrak{g}_{0}+A\left(c_{\alpha}\right)+\mathfrak{h}\left(c_{\alpha}\right)$.
We assume that both $\mathfrak{g}(\cdot)$ and $\mathfrak{h}(.)$
are defined for $\norm{c_{\alpha}}\leq a$ and that $\norm{\mathfrak{h}\left(c_{\alpha}\right)}\leq L\norm{c_{\alpha}}^{2}$.
We also assume that the first $\mathfrak{r}$ singular values of $A$
are at least as great as $\sigma>0$. Then the probability of $\norm{\mathfrak{g}\left(c_{\alpha}\right)}\leq L\epsilon$
relative to $\norm{c_{\alpha}}\leq\epsilon^{1/2}$ is less than or
equal to 
\[
D_{\alpha}!2^{\mathfrak{r}}L^{\mathfrak{r}}\epsilon^{\mathfrak{r}/2}\bigl/\sigma^{\mathfrak{r}}
\]
for $0<\epsilon^{1/2}\leq a$.\label{lem:2-transfer-of-volume-nonlinear}\end{lem}
\begin{proof}
If $\epsilon^{1/2}\leq a$ and $\norm{c_{\alpha}}\leq\epsilon^{1/2}$,
then $\norm{\mathfrak{h}\left(c_{\alpha}\right)}\leq L\epsilon$.
Therefore, $\norm{A\left(c_{\alpha}\right)+\mathfrak{g}_{0}}\leq2L\epsilon$.
The proof is completed by applying the previous lemma with $L\leftarrow2L$
and $a\leftarrow\epsilon^{1/2}$. 
\end{proof}
Applications of Lemmas \ref{lem:1-transfer-of-volume} and \ref{lem:2-transfer-of-volume-nonlinear}
will require us to get a handle on singular values. We will turn to
that in the next section. Before doing so, we recapitulate an elegant
argument of Sauer et al \cite{SauerYorkeCasdagli1991}. This argument,
although elementary, gives a good idea of the general approach when
the observation function is parametrized. 

Suppose $K$ is a smooth sub-manifold or even a fractal set of box
counting dimension $d$ and with compact closure that is a subset
of $\mathbb{R}^{n}$. Let the embedding dimension be $D>2d$. If $d\in\mathbb{Z}^{+}$,
we can take $D=2d+1$ as in Whitney's embedding theorem \cite{Hirsch1976}.
The following assumptions are made about the constant $C_{K}$:
\begin{description}
\item [{Assumption}] about $C_{K}$ (1): The set $K$ can be covered with
$C_{K}/\epsilon^{d}$ $\epsilon$-balls for any $\epsilon>0$.
\item [{Assumption}] about $C_{K}$ (2): The set $K\times K$ can be covered
with $C_{K}/\epsilon^{2d}$ $\epsilon$-balls for any $\epsilon>0$.
\end{description}
All balls are spherical.

A linear map from $\mathbb{R}^{n}$ to $\mathbb{R}^{D}$ can be written
as $F_{\alpha}(x)=\sum_{\alpha\in\mathcal{I}}c_{\alpha}\mathfrak{m}_{\alpha}x$,
where $\mathcal{I}$ is the index set $(i,j)$, $1\leq i\leq D$ and
$1\leq j\leq n$, and $\mathfrak{m}_{\alpha}$ is the matrix with
$1$ in the $i,j$th position if $\alpha=(i,j)$ and zero everywhere
else. Here $D_{\alpha}=nD$. We use $c_{\alpha}$ both to refer to
an entry of the vector $\left(c_{\alpha}\right)$ as in the definition
of $F_{\alpha}$ and to the vector as a whole as in $\norm{c_{\alpha}}$.
The slight ambiguity, which is resolved from context, is highly convenient.
In most instances, $c_{\alpha}$ refers to the vector as a whole. 

Define $G_{\alpha}(x,y)=F_{\alpha}(x)-F_{\alpha}(y)$. Assume $\norm{c_{\alpha}}\leq a_{0}$.
By compactness of the ball $\norm{c_{\alpha}}\leq a_{0}$, we may
assume the Lipshitz constant of $G_{\alpha}(x,y)$ (with respect to
$x,y$) to be bounded above by $L$. Define $\mathcal{K}(\delta)$
to be the set of all points $(x,y)\in K\times K$ satisfying $\norm{x-y}\geq\delta>0$.
Cover $\mathcal{K}(\delta)$ using $C_{K}/\epsilon^{2d}$ balls. Suppose
$G_{\alpha}(x,y)=0$ for some $(x,y)\in\mathcal{K}(\delta)$. Then
by the Lipshitz bound, we must have $\norm{G_{\alpha}(x,y)}\leq L\epsilon$
for $(x,y)$ that is a center of one of the $C_{K}/\epsilon^{2d}$
covering $\mathcal{K}(\delta)$.

The rest of the argument hinges on transferring the volume $\norm{G_{\alpha}(x,y)}\leq L\epsilon$
to parameter space. To do so, write $G_{\alpha}(x,y)$ in the form
\[
\left(\begin{array}{ccc}
\mathfrak{m}_{1,1}(x-y), & \mathfrak{m}_{1,2}(x-y), & \ldots\end{array}\right)(c_{\alpha})
\]
and observe that every column in the resulting matrix is in $\mathbb{R}^{D}$
and is all zeros except for a single entry equal to $\pi_{i}x-\pi_{i}y$,
where $\pi_{i}$ denotes the projection to the $i$th coordinate,
for some $i\in\set{1,\ldots,n}$. The first $D$ singular values of
that matrix are all equal to $\norm{x-y}\geq\delta$. Thus, we may
transfer volumes using Lemma \ref{lem:1-transfer-of-volume} and assert
that the probability of $G_{\alpha}(x,y)=0$ for some $(x,y)\in\mathcal{K}(\delta)$
relative to $\norm{c_{\alpha}}\leq a_{0}$ is at most 
\[
\frac{C_{K}}{\epsilon^{2d}}\times\frac{(nD)!L^{D}\epsilon^{D}}{\delta^{D}a_{0}^{D}}.
\]
By taking the limit $\epsilon\rightarrow0$ and because $D>2d$, it
follows that $G_{\alpha}(x,y)=0$ for some $(x,y)\in\mathcal{K}(\delta)$
only for a set of $c_{\alpha}$ of probability zero relative to the
ball $\norm{c_{\alpha}}\leq a_{0}$. By taking the union of the probability
zero sets with $\delta=1,\frac{1}{2},\frac{1}{2^{2}},\ldots$, we
may conclude that $G_{\alpha}(x,y)=0$ for some $(x,y)\in K\times K$,
$x\neq y$, only for a set of $c_{\alpha}$ of probability zero relative
to $\norm{c_{\alpha}}\leq a_{0}$. Equivalently, $x\rightarrow F_{\alpha}(x)$
is injective for $x\in K$ with probability one relative to the ball
$\norm{c_{\alpha}}\leq a_{0}$ in parameter space.

The argument derives its power by simply refining the cover of $\mathcal{K}(\delta)$
by using smaller and smaller $\epsilon$-balls. If $dF_{\alpha}(x,v)$
is the tangent map at $x$ applied to the tangent vector $v$, then
$dF_{\alpha}(x,v)=F_{\alpha}(v)$ because of the linearity of $F_{\alpha}(x)$
in $x$. If $K$ is a submanifold then $T_{1}K$ is the unit tangent
bundle consisting of points $(x,v)$ with $\norm v=1$. Injectivity
may be proved by considering $dF_{\alpha}(x,v)$ instead of $G_{\alpha}(x,y)$,
with Lemma \ref{lem:1-transfer-of-volume} invoked with $\sigma\leftarrow1$.

\section{Rank lemmas}

In proving a version of the Whitney embedding theorem, the argument
of Sauer et al \cite{SauerYorkeCasdagli1991} reviewed above writes
$G_{\alpha}(x,y)=\mathcal{M}.c_{\alpha}$ and relies on explicit knowledge
of singular values of $\mathcal{M}$. In general, singular values
of $\mathcal{M}$ cannot be obtained so explicitly. Instead, the approach
is to first argue that $\mathcal{M}$ has rank $D$ or greater for
every $(x,y)\in\mathcal{K}(\delta)$ and then observe that 
\[
\sigma_{\delta}=\min_{(x,y)\in\mathcal{K}(\delta)}\sigma_{D}\left(\mathcal{M}\right)>0
\]
because the $D$th singular value $\sigma_{D}(\mathcal{M})$ is continuous
in $x,y$ and $\mathcal{K}(\delta)$ is compact. The argument may
then be completed by applying Lemma \ref{lem:1-transfer-of-volume}
with $\sigma\leftarrow\sigma_{\delta}$.

To support such an argument, we give a few rank lemmas in this section.
The first two lemmas are from \cite{SauerYorkeCasdagli1991}. Rank
lemmas of this type are known in multivariate approximation theory
\cite{Madych2006}, although they are buried inside more sophisticated
results.

Suppose $z\in\mathbb{R}^{d}$. As noted already, the projection to
the $i$th coordinate is denoted by $\pi_{i}$. If $\alpha=(\alpha_{1},\ldots,\alpha_{d})$,
$\alpha_{i}\in\mathbb{Z}^{+}\cup\set 0$, is a multi-index, then $z^{\alpha}=\prod_{i=1}^{d}\left(\pi_{i}z\right)^{\alpha_{i}}$
as usual and $\abs{\alpha}=\sum_{i=1}^{d}\abs{\alpha_{i}}$. In later
arguments, it is essential to take the gradient of $z^{\alpha}$ with
respect to $z$. For notational convenience, we always denote $z^{\alpha}$
by $p_{\alpha}(z)$. The index set $\mathcal{I}_{D^{+}}$ is the set
of all $\alpha$ such that $\abs{\alpha}\leq D^{+}$. By elementary
combinatorics, the cardinality of $\mathcal{I}_{D^{+}}$ is $\binom{d+D^{+}}{D^{+}}$.

Suppose $z_{1},z_{2},\ldots,z_{D'}$ are distinct points in $\mathbb{R}^{d}$.
Then 
\begin{equation}
\left(\begin{array}{c}
p_{\alpha}(z_{1})\\
\vdots\\
p_{\alpha}(z_{D'})
\end{array}\right)\label{eq:vandermonde}
\end{equation}
denotes the multivariate Vandermonde matrix with the column index
$\alpha\in\mathcal{I}_{D^{+}}$ for some $D^{+}$. The dimension of
the matrix is $D'\times\abs{\mathcal{I}_{D^{+}}}$, where $\abs{\mathcal{I}_{D^{+}}}$
is the cardinality of $\mathcal{I}_{D^{+}}$.
\begin{lem}[\cite{SauerYorkeCasdagli1991}]
 For $\alpha\in\mathcal{I}_{D^{+}}$ and $D^{+}\geq D'-1$, the rank
of the Vandermonde matrix (\ref{eq:vandermonde}) is equal to the
number of its rows.\label{lem:3-vandermonde}\end{lem}
\begin{proof}
Following \cite{SauerYorkeCasdagli1991}, let $Q$ be a $d\times d$
orthogonal matrix drawn from the Haar measure. If $z_{1}$ and $z_{2}$
are distinct, then $\pi_{i}z_{1}\neq\pi_{i}z_{2}$ for any $i$ for
$Q$ outside a set of measure $0$. Therefore, we can find a $Q$
such that $\pi_{1}Qz_{1},\ldots,\pi_{1}Qz_{D'}$ are distinct. We
may interpolate arbitrary values at $z_{j}$ using a univariate polynomial
$\mathfrak{p}(\pi_{1}Qz)$ of degree $D'-1$. Because we can write
\[
\mathfrak{p}(\pi_{1}Qz)=\left(\begin{array}{c}
p_{\alpha}(z_{1})\\
\vdots\\
p_{\alpha}(z_{D'})
\end{array}\right)(c_{\alpha})
\]
for a suitable choice of $c_{\alpha}$, it follows that the rank of
(\ref{eq:vandermonde}) is equal to the number of its rows.
\end{proof}
Let 

\begin{equation}
\left(\begin{array}{c}
\nabla p_{\alpha}(z_{1})\\
\vdots\\
\nabla p_{\alpha}(z_{D'})
\end{array}\right)\label{eq:incomplete-Hermite}
\end{equation}
be the multivariate incomplete Hermite matrix at $z_{1},\ldots,z_{D'}$
and with $\alpha\in\mathcal{I}_{D^{+}}$. 
\begin{lem}[\cite{SauerYorkeCasdagli1991}]
 The rank of the incomplete Hermite matrix (\ref{eq:incomplete-Hermite})
is equal to the number of its rows if $D^{+}\geq D'$.\label{lem:4-incomplete-hermite}\end{lem}
\begin{proof}
Arguing as in the previous lemma, we may assume $\pi_{i}Qz_{1},\ldots,\pi_{i}Qz_{D'}$
to be distinct for $i=1,\ldots,d$. Following \cite{SauerYorkeCasdagli1991}
and assuming $Q$ to be the identity without loss of generality, we
may then find a polynomial $\mathfrak{p}_{i}(\pi_{i}z)$ of degree
$D'-1$ that interpolates the $i$th component of the prescribed gradients
at $z_{1},\ldots,z_{D'}$. We may then obtain the prescribed gradients
from $\mathfrak{p}(z)=\int\mathfrak{p}_{1}\,d\pi_{1}z+\cdots+\int\mathfrak{p}_{d}\,d\pi_{d}z$.
Thus, a linear combination of the columns of (\ref{eq:incomplete-Hermite})
can produce any prescribed gradients.
\end{proof}
To obtain prevalence results with a fixed observation function, Lemmas
\ref{lem:3-vandermonde} and \ref{lem:4-incomplete-hermite} need
to be combined into another lemma. Therefore, let 
\begin{equation}
\left(\begin{array}{c}
p_{\alpha}(z_{1})\\
\vdots\\
p_{\alpha}(z_{D'})\\
\nabla p_{\alpha}(z_{1})\\
\vdots\\
\nabla p_{\alpha}(z_{D'})
\end{array}\right)\label{eq:complete-Hermite}
\end{equation}
be the multivariate Hermite matrix at $z_{1},\ldots,z_{D'}$ and with
$\alpha\in\mathcal{I}_{D^{+}}$. 
\begin{lem}
The rank of the Hermite matrix (\ref{eq:complete-Hermite}) is equal
to the number of its rows if $D^{+}\geq2D'-1$.\label{lem:5-complete-Hermite}\end{lem}
\begin{proof}
Suppose function values as well as gradients are prescribed at $z_{1},\ldots,z_{D'}$.
We may obtain the prescribed gradients at $z_{2},\ldots,z_{D'}$ as
in the previous proof in the form $\mathfrak{p}(z)=\int\mathfrak{p}_{2}\,d\pi_{2}z+\cdots+\int\mathfrak{p}_{d}\,d\pi_{d}z$.
To obtain suitable function values as well as the $\pi_{1}$ component
of the gradients, we may take the polynomial $\mathfrak{p}(z)+\mathfrak{q}(\pi_{1}z)$
with $\mathfrak{q}$ being a suitable univariate Hermite interpolant
of degree $2D'-1$.
\end{proof}
A matrix $M$ is said to be circulant if its subsequent rows are obtained
by rotating the first row. If the number of columns is $\mathfrak{n}$
and the first row is $\mathfrak{a_{1}},\ldots,\mathfrak{a}_{n}$,
the second row must be $\mathfrak{a}_{n},\mathfrak{a}_{1},\ldots,\mathfrak{a}_{n-1}$.
The following lemma about circulant matrices will be used in the next
section to refine the discussion of \cite{SauerYorkeCasdagli1991}.
\begin{lem}
Let $M$ be a $\mathfrak{m}\times D'$ circulant matrix whose first
row is $1,0^{j_{1}},-1,0^{j_{2}}$, where $0^{j_{1}}$ is $0$ repeated
$j_{1}$ times. The rank of $M$ is equal to $\mathfrak{m}$ if $\mathfrak{m}\leq\ceil{D'/2}$.\label{lem:6-circulant}\end{lem}
\begin{proof}
We must have $j_{1}+j_{2}=D'-2$. Either $j_{1}$ or $j_{2}$ must
be less than or equal to $(D'-2)/2$. Because they are both integers,
either $j_{1}$ or $j_{2}$ must be $\leq\floor{\frac{D'-2}{2}}$.
Without loss of generality, we assume $j_{1}\leq\floor{D'/2}-1$.
As the rows are rotated, the $-1$ appears in column $j_{1}+k+1$
for $k=1,\ldots,\mathfrak{m}$. The columns do not wrap around because
\[
j_{1}+\mathfrak{m}+1\leq\floor{D'/2}-1+\ceil{D'/2}+1\leq D'.
\]
All those columns are linearly independent.
\end{proof}
The final rank lemma is obvious from elementary linear algebra. We
state it explicitly because it is invoked often and has a key position
in the framework of \cite{SauerYorkeCasdagli1991}. For the most part,
the lemma is invoked silently.
\begin{lem}
If the rank of the matrix $B$ is equal to the number of its rows,
the rank of the product $AB$ is equal to the rank of $A$.\label{lem:7-linalg}
\end{lem}

\section{Review of Sauer et al \cite{SauerYorkeCasdagli1991}}

In this section, we review the main results and proofs of \cite{SauerYorkeCasdagli1991}.
Our aim is two-fold. The review helps us prepare the ground for our
results about prevalence with a fixed observation map. Second, we
point out and fix an error in \cite{SauerYorkeCasdagli1991}, while
presenting the proof with greater formal precision and completeness.
The error in \cite{SauerYorkeCasdagli1991} is a minor one relative
to the depth of ideas found in that paper. We also point out errors
and gaps in earlier mathematical treatments that are much more serious.

Let $\phi:\mathbb{R}^{n}\rightarrow\mathbb{R}^{n}$ be a diffeomorphism
that is at least $C^{2}$. We adopt the following convention:
\begin{description}
\item [{Convention}] about $x,y$: If $x_{1}$ is a point in $\mathbb{R}^{n}$,
then $x_{2}=\phi(x_{1})$, $x_{3}=\phi(x_{2})$, and so on. Similarly,
$y_{2}=\phi(y_{1})$, $y_{3}=\phi(y_{2})$, and so on. It must be
noted that this convention does not apply to $z$. For example, $z_{1},\ldots,z_{D'}$
are any distinct points in Lemma \ref{lem:3-vandermonde}.
\end{description}
The observation function is assumed to be the (at least twice continuously
differentiable) function $o:\mathbb{R}^{n}\rightarrow\mathbb{R}$,
which maps every state vector to a real number. If the state vector
is $x_{1}$, the corresponding delay vector is 
\[
F_{0}(x_{1})=\left(\begin{array}{c}
o(x_{1})\\
\vdots\\
o(x_{D})
\end{array}\right),
\]
where $D$ will be referred to as the embedding dimension. 

Let $K\subset\mathbb{R}^{n}$ be a possibly fractal set of box counting
dimension $d$. The set $K$ is assumed to be compact. The delay mapping
$F_{0}$ restricted to $K$ may not be injective. To examine the injectivity
more generally, we perturb the observation function to 
\[
o(x)+\sum_{\alpha\in\mathcal{I}_{2D-1}}c_{\alpha}p_{\alpha}(x)
\]
and examine injectivity in the ball $\norm{c_{\alpha}}\leq a_{0}$
with $a_{0}>0$ and fixed. The perturbed delay vector becomes
\[
F_{\alpha}(x)=F_{0}(x)+\left(\begin{array}{c}
p_{\alpha}(x_{1})\\
\vdots\\
p_{\alpha}(x_{D})
\end{array}\right)(c_{\alpha}),
\]
with $\alpha$ ranging over $\mathcal{I}_{2D-1}$. We use $F_{\alpha}$
instead of $F_{c_{\alpha}}$to denote the delay vector for simplicity
and without risk of confusion. The two assumptions about $C_{K}$
made in the previous section are carried forward.
\begin{thm}[\cite{SauerYorkeCasdagli1991}]
If $D>2d$ and $\phi$ has finitely many periodic points $x$ of
periods less than $2D$, the delay mapping $x\rightarrow F_{\alpha}(x)$
is injective for $x\in K$ for a set of $c_{\alpha}$ of probability
$1$ relative to $\norm{c_{\alpha}}\leq a_{0}$. \label{thm:8-SYC-injectivity}
\end{thm}
Theorem \ref{thm:8-SYC-injectivity} is less general than corresponding
statements in \cite{SauerYorkeCasdagli1991}. Our aim is to exhibit
techniques while forsaking generality. The manner in which more general
statements can be obtained is discussed later.
\begin{proof}
Define $G_{\alpha}(x_{1},y_{1})=F_{\alpha}(x_{1})-F_{\alpha}(y_{1})$.
We then have $G_{\alpha}(x_{1},y_{1})=F_{0}(x_{1})-F_{0}(y_{1})+\mathcal{M}(c_{\alpha})$,
where 
\[
\mathcal{M}=J\mathcal{V},\:\:J=\left(\begin{array}{cccccc}
1 &  &  & -1\\
 & \ddots &  &  & \ddots\\
 &  & 1 &  &  & -1
\end{array}\right),\:\:\mathcal{V}=\left(\begin{array}{c}
p_{\alpha}(x_{1})\\
\vdots\\
p_{\alpha}(x_{D})\\
p_{\alpha}(y_{1})\\
\vdots\\
p_{\alpha}(y_{D})
\end{array}\right).
\]
Here $J$ is $D\times2D$ and $\mathcal{V}$ is $2D\times D_{\alpha}$,
where $D_{\alpha}$ is the cardinality of $\mathcal{I}_{2D-1}$. The
proof turns on the determination of the rank of $\mathcal{M}$. If
$x_{i}$ and $y_{i}$, $1\leq i\leq D$, are $2D$ distinct points,
we may apply Lemmas \ref{lem:3-vandermonde} and \ref{lem:7-linalg}
and immediately conclude that the rank of $\mathcal{M}$ is $D$.
However, if not all points are distinct, the rank of $\mathcal{V}$
is obviously not equal to the number of rows. Several cases need to
be considered to determine the rank of $\mathcal{M}.$

\emph{Case 1:} both $x_{1}$ and $y_{1}$ are periodic of period less
than $2D$ with $x_{1}\neq y_{1}$. The set of such pairs $(x_{1},y_{1})$
is finite (by assumption) and will be denoted by $\mathcal{K}_{1}$.
There are two subcases.

\emph{Case 1.1:} $x_{1}$ and $y_{1}$ lie on distinct orbits. If
so $\mathcal{M}$ can be written in a compressed form as $\mathcal{M}=J_{c}\mathcal{V}_{c}$
with 
\[
J_{c}=\left(\begin{array}{cc}
C_{1} & C_{2}\end{array}\right),\:\:\mathcal{V}_{c}=\left(\begin{array}{c}
p_{\alpha}(x_{1})\\
\vdots\\
p_{\alpha}(x_{\mathfrak{p}})\\
p_{\alpha}(y_{1})\\
\vdots\\
p_{\alpha}(y_{\mathfrak{q}})
\end{array}\right),
\]
where $\mathfrak{p,q}$ are the periods of $x_{1},y_{1}$ (or $D$
if the periods are greater than $D$), respectively. Further, $C_{1}$
is a $D\times\mathfrak{p}$ circulant matrix with first row $1,0,\ldots$
and $C_{2}$ is a $D\times\mathfrak{q}$ circulant matrix with first
row $-1,0,\ldots$. The rank of $\mathcal{V}_{c}$ is equal to the
number of its rows by Lemma \ref{lem:3-vandermonde} and $J_{c}$
is nonzero. Therefore, we may assert that the rank of $\mathcal{M}$
is $1$ or greater.

\emph{Case 1.2:} $x_{1}$ and $y_{1}$ lie on the same periodic orbit.
In this case, we may write 
\[
J_{c}=C_{1},\:\:\mathcal{V}_{c}=\left(\begin{array}{c}
p_{\alpha}(x_{1})\\
\vdots\\
p_{\alpha}(x_{\mathfrak{p}})
\end{array}\right),
\]
where $\mathfrak{p}$ is the period of $x_{1}$, $C_{1}$ is a $D\times\mathfrak{p}$
circulant matrix whose first row is of the form $1,0,\ldots,0,-1,0,\ldots.0$.
Again, we conclude that the rank of $\mathcal{M}$ is greater than
$1$.

Suppose $G_{\alpha}(x_{1},y_{1})=0$ for some $(x_{1},y_{1})\in\mathcal{K}_{1}$.
Then $\mathcal{M}c_{\alpha}=0$ and $c_{\alpha}$ must lie on a hyperplane
of co-dimension $1$ or greater. Because $\mathcal{K}_{1}$ is finite,
we may assert $G_{\alpha}(x_{1},y_{1})\neq0$ for all $(x_{1},y_{1})\in\mathcal{K}_{1}$
with probability $1$ relative to the ball $\norm{c_{\alpha}}\leq a_{0}$.
Case 1 is now complete.

\emph{Case 2:} Define $\mathcal{K}_{2}(\delta)$ to be the set of
all $(x_{1},y_{1})\in K\times K$ satisfying 
\begin{enumerate}[nolistsep]
\item $\norm{x_{1}-y_{1}}\geq\delta$,
\item $\mathrm{dist}\left((x_{1},y_{1}),\mathcal{K}_{1}\right)\geq\delta$.
All distances in this paper use the $L_{2}$ or spectral norm.
\end{enumerate}
The matrix $\mathcal{M}$ has a rank equal to $D$ for each point
in $\mathcal{K}_{2}(\delta)$, as we will prove by breaking up case
2 into subcases.

\emph{Case 2.1: }$x_{1},\ldots,x_{D},y_{1},\ldots,y_{D}$ are $2D$
distinct points. In this case, $\mathcal{M}=J\mathcal{V}$ has rank
equal to $D$ as noted at the beginning of the proof.

\emph{Case 2.2: }$x_{1},\ldots,x_{D}$ are distinct, $y_{1},\ldots,y_{D}$
are distinct, and neither $x_{1}$ nor $y_{1}$ is a periodic point
of period less than $2D$, but $y_{1}=x_{j}$ or $x_{1}=y_{j}$ for
$j\in\set{2,\ldots,D}$. Without loss of generality, we assume $y_{1}=x_{j}$.

In this case, the compressed form is $\mathcal{M}=J_{c}\mathcal{V}_{c}$
with 
\[
J_{c}=(C_{1}),\:\:\mathcal{V}_{c}=\left(\begin{array}{c}
p_{\alpha}(x_{1})\\
\vdots\\
p_{\alpha}(x_{D+j-1})
\end{array}\right),
\]
where $C_{1}$ is $D\times(D+j-1)$ circulant matrix with first row
equal to $1,0^{j-2},-1,0^{D-1}$. The $-1$ does not wrap around and
the rank of $C_{1}$ and therefore of $\mathcal{M}$ is $D$.

\emph{Case 2.3: }$x_{1}$ periodic of period less than $2D$ and $y_{1}$
not so (or vice versa, which may be ignored without loss of generality).
In this case, the compressed form is $\mathcal{M}=J_{c}\mathcal{V}_{c}$
with 
\[
J_{c}=(C_{1},C_{2}),\:\:\mathcal{V}_{c}=\left(\begin{array}{c}
p_{\alpha}(x_{1})\\
\vdots\\
p_{\alpha}(x_{\mathfrak{p}})\\
p_{\alpha}(y_{1})\\
\vdots\\
p_{\alpha}(y_{D})
\end{array}\right),
\]
where $\mathfrak{p}$ is the period of $x_{1}$, $C_{1}$ is a $D\times\mathfrak{p}$
circulant matrix with first row $1,0,\ldots$, and $C_{2}$ is a $D\times D$
circulant matrix with first row $-1,0,\ldots$ The column rank of
$C_{2}$ is equal to $D$ and therefore the rank of $\mathcal{M}$
is also $D$.

We can now complete case 2 as follows. Suppose $G_{\alpha}(x_{1},y_{1})=0$
for some $(x_{1},y_{1})\in\mathcal{K}_{2}(\delta)$. By assumption
(2) about $C_{K}$, cover $\mathcal{K}_{2}(\delta)$ with $C_{K}/\epsilon^{2d}$
or fewer $\epsilon$-balls for $\epsilon>0$. At this point, we introduce
an assumption about $L$:
\begin{description}
\item [{Assumption}] about $L$ (1): The Lipshitz constant of $G_{\alpha}(x_{1},y_{1})$
with respect to $(x_{1},y_{1})\in K\times K$ and with $\norm{c_{\alpha}}\leq a_{0}$
is bounded by $L$. The existence of $L$ is a consequence of the
compactness of $K\times K$, the compactness of $\norm{c_{\alpha}}\leq a_{0}$,
and the differentiability assumption about the observation function
$o$ and the diffeomorphism $\phi$.
\end{description}
It then follows that if $G_{\alpha}(x_{1},y_{1})=0$ at some point
$(x_{1},y_{1})\in\mathcal{K}_{2}(\delta)$, then $\norm{G_{\alpha}(x_{1},y_{1})}=\norm{\mathcal{M}(c_{\alpha})}\leq L\epsilon$
at the center of one of the $\epsilon$-balls covering $\mathcal{K}_{2}(\delta)$.
Define 
\[
\sigma_{\delta}=\min_{(x_{1},y_{1})\in\mathcal{K}_{2}(\delta)}\sigma_{D}\left(\mathcal{M}\right).
\]
By compactness of $\mathcal{K}_{2}(\delta)$, $\sigma_{\delta}$ exists
and is positive. By the transfer of volume Lemma \ref{lem:1-transfer-of-volume},
which is applied with $\mathfrak{r}\leftarrow D$, the probability
of $\norm{G_{\alpha}(x_{1},y_{1})}\leq L\epsilon$ relative to the
ball $\norm{c_{\alpha}}\leq a_{0}$ at a point $(x_{1},y_{1})\in\mathcal{K}_{2}(\delta)$
is upper bounded by 
\[
\frac{D_{\alpha}!L^{D}\epsilon^{D}}{\sigma_{\delta}^{D}a_{0}^{D}}.
\]
Because $\mathcal{K}_{2}(\delta)$ can be covered with $C_{K}/\epsilon^{2d}$
or fewer $\epsilon$-balls, the probability that $G_{\alpha}(x_{1},y_{1})=0$
for some $(x_{1},y_{1})\in\mathcal{K}_{2}(\delta)$ is upper bounded
by 
\[
\frac{C_{K}}{\epsilon^{2d}}\times\frac{D_{\alpha}!L^{D}\epsilon^{D}}{\sigma_{\delta}^{D}a_{0}^{D}}.
\]
Because $D>2d$ and by taking $\epsilon\rightarrow0$, we conclude
that the probability of $G_{\alpha}(x_{1},y_{1})=0$ for some $(x_{1},y_{1})\in\mathcal{K}_{2}(\delta)$
relative to $\norm{c_{\alpha}}\leq a_{0}$ is one. Case 2 is now complete.

To complete the proof of injectivity, take the union of the measure
zero sets in case 2 with $\delta=1,\frac{1}{2},\frac{1}{2^{2}},\ldots$
and the measure zero set in case 1. Outside of that measure $0$ subset
of the ball $\norm{c_{\alpha}}\leq a_{0}$, we have $G_{\alpha}(x_{1},y_{1})\neq0$
for $(x_{1},y_{1})\in K\times K$ and $x_{1}\neq y_{1}$.
\end{proof}
The ideas in the proof presented above are from \cite{SauerYorkeCasdagli1991},
although our presentation is more precise and formally complete. Theorem
\ref{thm:8-SYC-injectivity} makes an assumption on periodic points
of period $<2D$ and not $\leq D$ as in \cite{SauerYorkeCasdagli1991}.
To see why the more stringent assumption is needed, we turn to \cite[p. 611, case 3]{SauerYorkeCasdagli1991}.
The case ``$x$ and $y$ are not both periodic with period $\leq w$''
is considered ($w$ is $D$ in our notation) and it is stated that
$J_{xy}$ (which is $J_{c}$ in our notation) is triangular of rank
$D$. Unfortunately, that statement is not correct. 

To understand why that statement is not true, assume $D=6$. Suppose
$x_{1}$ is a periodic point of period $8>D$ and that $y_{1}=x_{5}$.
Then $J_{c}$ will be a $6\times8$ circulant matrix which looks as
follows:
\[
\left(\begin{array}{cccccccc}
1 & 0 & 0 & 0 & -1 & 0 & 0 & 0\\
0 & 1 & 0 & 0 & 0 & -1 & 0 & 0\\
0 & 0 & 1 & 0 & 0 & 0 & -1 & 0\\
0 & 0 & 0 & 1 & 0 & 0 & 0 & -1\\
-1 & 0 & 0 & 0 & 1 & 0 & 0 & 0\\
0 & -1 & 0 & 0 & 0 & 1 & 0 & 0
\end{array}\right).
\]
Evidently, the rank of this matrix is $4<D$. 

The easiest way to fix the minor error is to assume the number of
periodic points of period $<2D$ to be finite as we have done. However,
Sauer et al \cite{SauerYorkeCasdagli1991} place conditions on the
box counting dimension of the set of periodic points of period $p$.
The conditions involving quantities such as $\mathrm{rank}(BC_{pq}^{w})$
are not easy to interpret and it is unclear what they mean. The basic
idea of assuming a bound on the box counting dimension of periodic
points of a certain period is a sound one. It can be developed fully
using Lemma \ref{lem:6-circulant} about the rank of circulant matrices
and variations of that lemma. We have not done so for two reasons.
The proof becomes a great deal more complicated, and at this point
having a clear and complete account of the main ideas appears more
important than a slightly more general theorem. Additionally, if the
box counting dimension of the set of periodic points is greater than
$1$, then $1$ will be a characteristic multiplier that is repeated
more than once, which is excluded in the immersivity theorem.

The gaps in \cite{Takens1981} and \cite{Aeyels1980} are much less
minor. In \cite{Takens1981}, it is assumed that the delay map is
an embedding in some neighborhood of the periodic points. The proof
of that assumption is unlikely to be as straightforward as assumed.
Even granting that assumption, the argument for transversality \cite[p. 371]{Takens1981}
appears incomplete. In particular, it does not consider the possibility
that perturbing the delay map of $x$ may also perturb the delay map
of $x'$, for example, when $x'=\phi(x)$ and the orbits of $x,x'$
overlap. There are yet other aspects of the proof we were not able
to verify. For example, \cite[p. 370, case iii]{Takens1981} seems
to require $x$ to be close to a periodic point and $x'$ to be away
from a periodic point. It is then asserted that $x,\ldots,\phi^{2m}(x),x',\ldots,\phi^{2m}(x')$
are distinct. How could that be true if $x$ is a fixed point? How
is the possibility $x'=\phi(x)$ handled?

The gaps in \cite{Aeyels1980} also occur in handling overlaps of
orbits and periodic points. The main argument \cite[p. 598]{Aeyels1980}
entirely ignores the possibility that orbits of $x^{\ast}$ and $y^{\ast}$
may overlap. Further, it is suggested that difficulties associated
with fixed points can be handled by adjusting the delays but no details
are provided about carrying out that suggestion. 

Going back to the work of Sauer et al \cite{SauerYorkeCasdagli1991},
a point in our proof of Theorem \ref{thm:8-SYC-injectivity} is worth
calling to attention. In the proof, $\mathcal{K}_{2}(\delta)$ is
covered with $\epsilon$-balls and it is assumed that every ball center
is in $\mathcal{K}_{2}(\delta)$. It is not sufficient to start with
any cover of $K\times K$ because a ball center can be arbitrarily
close to the diagonal or to a pair of periodic points and $\sigma_{D}(\mathcal{M})$
may become arbitrarily small. 

If we say that a certain compact set $\mathfrak{s}$ is covered by
a certain number of $\epsilon$-balls, it is assumed that each ball
has a center that lies in $\mathfrak{s}$. That assumption comes up
repeatedly in the proof of immersivity, which we now turn to. Once
again all the ideas are from \cite{SauerYorkeCasdagli1991}. Here
$K$ is assumed to be a smooth, closed, and compact submanifold of
dimension $d$ and $T_{1}K$ denotes its unit tangent bundle. If $x\in K$
and $v$ is tangent to $K$ at $x$, then $(x,v)\in T_{1}K$ if and
only if $\norm v=1$.
\begin{thm}
\cite{SauerYorkeCasdagli1991} If $D>2d-1$, $K$ is invariant under
$\phi$, $\phi$ has finitely many points $x\in K$ of period less
than $D$, and all characteristic multipliers of each of those points
are distinct, then $x\rightarrow F_{\alpha}(x)$ is immersive over
$K$ with probability $1$ relative to the ball $\norm{c_{\alpha}}\leq a_{0}$.\label{thm:9-SYC-immersivity}\end{thm}
\begin{proof}
If $x\rightarrow F_{\alpha}(x)$ and $v$ is a tangent vector to $K$
at $x$, then we denote the vector that $v$ is mapped to by $dF_{\alpha}(x,v)$.
The following convention about $v$ is an extension of the convention
about $x,y$ explained earlier.
\begin{description}
\item [{Convention}] about $v$: If $v_{1}$ is tangent to $K$ at $x_{1}$,
then $v_{2}=\frac{\partial\phi}{\partial x}\Bigl|_{x_{1}}v_{1}$,
$v_{3}=\frac{\partial\phi}{\partial x}\Bigl|_{x_{2}}v_{2}$, and so
on. Because $\phi$ is a diffeomorphism, $v_{i}$ are all nonzero
like $v_{1}$.
\end{description}
We write $dF_{\alpha}(x_{1},v_{1})=dF_{0}(x_{1},v_{1})+\mathcal{N}(c_{\alpha})$,
where 
\[
\mathcal{N}=J\mathcal{H},\:\:J=\left(\begin{array}{ccc}
v_{1}^{T}\\
 & \ddots\\
 &  & v_{D}^{T}
\end{array}\right),\:\:\mathcal{H}=\left(\begin{array}{c}
\nabla p_{\alpha}(x_{1})\\
\vdots\\
\nabla p_{\alpha}(x_{D})
\end{array}\right).
\]
The proof will turn on the rank of $\mathcal{N}=J\mathcal{H}$. If
$x_{1},\ldots,x_{D}$ are distinct, the rank of $\mathcal{N}$ is
$D$ because the rank of $\mathcal{H}$ is equal to the number of
its rows by Lemma \ref{lem:4-incomplete-hermite} and the rank of
$J$ is obviously $D$.

To study the rank of $\mathcal{N}$, it is useful to define the following
disjoint sets of $T_{1}K$.
\begin{itemize}
\item $\mathcal{K}_{1}$ is the set of all $(x_{1},v_{1})$ such that $x_{1}$
is a periodic point of period less than $D$ and $v_{1}$ is an eigenvector
of the periodic point $x_{1}$. By eigenvector of a periodic point,
we mean an eigenvector of the corresponding monodromy matrix. 
\item $\mathcal{K}_{2}(\delta)$ is the set of all $(x_{1},v_{1})$ such
that $x_{1}$ is a periodic point of period less than $D$ and $v_{1}$
is a linear combination of two eigenvectors of $x_{1}$. It is also
required that 
\[
\mathrm{dist}((x_{1},v_{1}),\mathcal{K}_{1})\geq\delta.
\]
We will denote $\mathcal{K}_{2}(0)$, where this last condition is
not operative, by $\mathcal{K}_{2}$. Evidently, $\mathcal{K}_{1}$
is a subset of $\mathcal{K}_{2}$.
\item In general, $\mathcal{K}_{r}(\delta)$, where $r=2,\ldots,d$, is
defined as the set of $(x_{1},v_{1})\in T_{1}K$ such that $x_{1}$
is a periodic point of period $D$ or less and $v_{1}$ is a linear
combination of $r$ eigenvectors of the periodic point $x_{1}$.It
is also required that 
\[
\mathrm{dist}((x_{1},v_{1}),\mathcal{K}_{r-1})\geq\delta.
\]
We will denote $\mathcal{K}_{r}(0)$, where this last condition is
not operative, by $\mathcal{K}_{r}$. Evidently, $\mathcal{K}_{r-1}$
is a subset of $\mathcal{K}_{r}$.\\
 This sequence of cases stops at $r=d$ and does not go up to $r=n$
because we are only interested in those eigenvectors of the periodic
point $x_{1}$that are also tangent to $K$. The assumption about
the invariance of $K$ is used here. 
\item The final case is $\mathcal{K}_{D}(\delta)$ which consists of all
points $(x_{1},v_{1})\in T_{1}K$ such that $x_{1}$is not periodic
of period less than $D$ and the distance to $\mathcal{K}_{d}$ is
$\geq\delta$. 
\end{itemize}

The final case $\mathcal{K}_{D}(\delta)$ is the easiest to handle.
In this case, $x_{1},\ldots,x_{D}$ are distinct and the rank of $\mathcal{N}$
is $D$ as already mentioned. 

In the case $\mathcal{K}_{r}(\delta)$, the rank of $\mathcal{N}$
is in fact $r$ or greater. To verify, suppose $(x_{1},v_{1})\in\mathcal{K}_{r}(\delta)$.
Assume $v_{1}=\mathfrak{u}_{1}+\cdots+\mathfrak{u}_{r}$, where $\mathfrak{u}_{i}$
are eigenvectors at the periodic point $x_{1}$. Assume $v_{2}=\mathfrak{v}_{1}+\cdots+\mathfrak{v}_{r}$,
where $\mathfrak{v}_{i}$ are eigenvectors at point $x_{2}$ obtained
by pushing $\mathfrak{u}_{i}$ along with the map $\phi$. Likewise,
if $x_{1}$ is of period $p$, assume that $v_{p}=\mathfrak{w}_{1}+\cdots+\mathfrak{w}_{r}$.

Then the compressed form of $\mathcal{N}$ is $\mathcal{N}=J_{c}\mathcal{H}_{c}$
with 
\[
J_{c}=\left(\begin{array}{cccc}
\mathfrak{u}_{1}^{T}+\cdots+\mathfrak{u}_{r}^{T}\\
 & \mathfrak{v}_{1}^{T}+\cdots+\mathfrak{v}_{r}^{T}\\
 &  & \ddots & \mathfrak{w}_{1}^{T}+\cdots+\mathfrak{w}_{r}^{T}\\
\lambda_{1}\mathfrak{u}_{1}^{T}+\cdots+\lambda_{r}\mathfrak{u}_{r}^{T}\\
 & \ddots\\
 &  &  & \lambda_{1}\mathfrak{w}_{1}^{T}+\cdots+\lambda_{r}\mathfrak{w}_{r}^{T}\\
\lambda_{1}^{2}\mathfrak{u}_{1}^{T}+\cdots+\lambda_{r}^{2}\mathfrak{u}_{r}^{T}\\
 & \ddots
\end{array}\right),
\]
where $\lambda_{1},\ldots,\lambda_{r}$ are characteristic multipliers
and the pattern is repeated until $D$ rows are obtained, and 
\[
\mathcal{H}_{c}=\left(\begin{array}{c}
\nabla p_{\alpha}(x_{1})\\
\vdots\\
\nabla p_{\alpha}(x_{p})
\end{array}\right).
\]
The rank of $\mathcal{H}_{c}$ is equal to the number of its rows
by Lemma \ref{lem:4-incomplete-hermite}. The rank of $J_{c}$ is
$\min(rp,D)$ because the Vandermonde matrix 
\[
\left(\begin{array}{cccc}
1 & 1 & \ldots & 1\\
\lambda_{1} & \lambda_{2} & \ldots & \lambda_{r}\\
 & \vdots\\
\lambda_{1}^{r-1} & \lambda_{2}^{r-1} &  & \lambda_{r}^{r-1}
\end{array}\right)
\]
has full rank, the $\lambda_{i}$ being distinct by assumption. Therefore,
the rank of $\mathcal{N}$ is $r$ or greater for each $(x_{1},y_{1})\in\mathcal{K}_{r}(\delta)$.

To complete the proof, we note that $\mathcal{K}_{r}$ is of dimension
$r-1$ for $r=1,\ldots,d$ and that $T_{1}K$ is of dimension $2d-1$.
A new assumption about $C_{K}$ is useful.
\begin{description}
\item [{Assumption}] about $C_{K}$ (3): It is assumed that $\mathcal{K}_{r}$
can be covered with $C_{K}/\epsilon^{r-1}$ $\epsilon$-balls for
$r=1,\ldots,d$. It is assumed that $T_{1}K$ and therefore $\mathcal{K}_{D}(\delta)$
can be covered with $C_{K}/\epsilon^{2d-1}$ $\epsilon$-balls. 
\end{description}

We also extend the assumption about the Lipshitz bound $L$.
\begin{description}
\item [{Assumption}] about $L$ (2): It is assumed that the Lipshitz constant
of $dF_{\alpha}$ with respect to $(x_{1},v_{1})\in T_{1}K$ for $\norm{c_{\alpha}}\leq a_{0}$
is upper bounded by $L$. This assumption too may be verified using
compactness like the first assumption about $L$.
\end{description}

The proof may now be completed easily. Suppose $dF_{\alpha}(x_{1},v_{1})=0$
for some $(x_{1},v_{1})\in\mathcal{K}_{r}(\delta)$. Then $\norm{dF_{\alpha}(x_{1},v_{1})}\leq L\epsilon$
at the center of one of the $C_{K}/\epsilon^{r-1}$ balls covering
$\mathcal{K}_{r}(\delta)$. By the transfer of volume Lemma \ref{lem:1-transfer-of-volume},
the probability of such an event is upper bounded by 
\[
\frac{C_{K}}{\epsilon^{r-1}}\times\frac{D_{\alpha}!L^{r}\epsilon^{r}}{\sigma_{\delta}^{r}a_{0}^{r}},
\]
where $\sigma_{\delta}=\min\sigma_{r}(\mathcal{N})$ over $(x_{1},v_{1})\in\mathcal{K}_{r}(\delta)$.
The probability evidently goes to $0$ as $\epsilon\rightarrow0$
leaving us with a measure zero set of $c_{\alpha}$ where $F_{\alpha}$
is not immersive at some point in $\mathcal{K}_{r}(\delta)$ for $r=2,\ldots,d$.
The sets $\mathcal{K}_{1}$ and $\mathcal{K}_{D}(\delta)$ are handled
similarly.
\end{proof}
Theorem \ref{thm:9-SYC-immersivity} assumes $K$ to be a closed and
compact submanifold. That assumption implies $\mathcal{K}_{D}(\delta)$
to be compact. If $\mathcal{K}_{D}(\delta)$ is compact, we may conclude
that $\min\sigma_{D}(\mathcal{N})$ over $(x_{1},v_{1})\in\mathcal{K}_{D}(\delta)$
exists and is positive. The assumptions on $K$ can be reduced. However,
the technicalities that arise (see \cite{Hirsch1976}) are extraneous
to the main ideas in this paper.

\section{Perturbing the dynamical system}

Let $\phi:\mathbb{R}^{d}\rightarrow\mathbb{R}^{d}$ be a diffeomorphism,
which is as before but with $n=d$. Let $\psi(x)$ denote $\frac{\partial\phi}{\partial x}$.
The vector in $\mathbb{R}^{d}$ with first component $1$ and the
others zero is denoted by $\mathbf{e}_{1}$. The perturbed dynamical
system is
\[
\phi_{\alpha}(x)=\phi(x)+\mathbf{e}_{1}\left(p_{\alpha}(x)\right)(c_{\alpha})
\]
with $\alpha\in\mathcal{I}_{2D-1}$, where $D$ is the embedding dimension.
It may be noted we are only perturbing the first coordinate of $\phi$.
Because the observation function will be assumed to be $o=\pi_{1}$,
it is enough to perturb only the first coordinate.

The delay vector under $\phi$ is 
\[
F_{0}(x_{1})=\left(\begin{array}{c}
\pi_{1}x_{1}\\
\vdots\\
\pi_{1}x_{D}
\end{array}\right).
\]

\begin{description}
\item [{Convention}] about $\tilde{x}$: It is assumed that $\tilde{x}_{1}=x_{1}$
. Thereafter, it is assumed that $\tilde{x}_{2}=\phi_{\alpha}(\tilde{x}_{1})$,
$\tilde{x}_{3}=\phi_{\alpha}(\tilde{x}_{2}),$and so on.
\end{description}
The delay vector under $\phi_{\alpha}$ is therefore
\[
F_{\alpha}(x_{1})=\left(\begin{array}{c}
\pi_{1}\tilde{x}_{1}\\
\vdots\\
\pi_{1}\tilde{x}_{D}
\end{array}\right).
\]
It is worthy of notice that $\phi_{\alpha}$ perturbs only the first
component of $\phi$. Because the delay vector is built up using $\pi_{1}$,
$\phi_{\alpha}$ must perturb the first component. If not, the perturbation
may not propagate to the delay vector at all. It turns out that perturbing
only the first component is also sufficient to obtain a prevalence
theorem.

Our first task is to express $F_{\alpha}$ as a perturbation of $F_{0}$.
That can be done by simply iterating the definition of $\phi_{\alpha}$:
\begin{align*}
\tilde{x}_{1} & =x_{1}\\
\tilde{x}_{2} & =x_{2}+\mathbf{e}_{1}\left(p_{\alpha}(x_{1})\right)(c_{\alpha})\\
\tilde{x}_{3} & =x_{3}+\mathbf{e}_{1}\left(p_{\alpha}(x_{2})\right)(c_{\alpha})+\psi(x_{2})\mathbf{e}_{1}\left(p_{\alpha}(x_{1})\right)(c_{\alpha})+\mathcal{O}(c_{\alpha}^{2}).
\end{align*}
Above and later, $\mathcal{O}(c_{\alpha}^{2})$ is the same as $\mathcal{O}\left(\norm{c_{\alpha}}^{2}\right)$.
By following the pattern, we obtain
\begin{equation}
\tilde{x}_{j}=x_{j}+\mathbf{e}_{1}\left(p_{\alpha}(x_{j-1})\right)(c_{\alpha})+\rho_{j-1}(x_{2},\ldots,x_{j-1},p_{\alpha}(x_{1}),\ldots,p_{\alpha}(x_{j-2}))(c_{\alpha})+\mathcal{O}(c_{\alpha}^{2})\label{eq:rho-definition}
\end{equation}
for $j=2,\ldots,D$. Here it is important to note that $\rho_{j-1}$
is linear in $p_{\alpha}(x_{1}),\ldots,p_{\alpha}(x_{j-2})$. For
brevity, we will rewrite (\ref{eq:rho-definition}) as 
\begin{equation}
\tilde{x}_{j}=x_{j}+\mathbf{e}_{1}\left(p_{\alpha}(x_{j-1})\right)(c_{\alpha})+\rho_{j-1}(c_{\alpha})+\mathcal{O}(c_{\alpha}^{2}).\label{eq:rho-defn-brief}
\end{equation}
We then get 
\[
F_{\alpha}(x_{1})=F_{0}(x_{1})+\left(\begin{array}{c}
0\\
V(x_{1})
\end{array}\right)(c_{\alpha})+\mathcal{O}(c_{\alpha}^{2}),
\]
with the matrix $V(x_{1})$ defined by 
\[
V(x_{1})=\left(\begin{array}{c}
p_{\alpha}(x_{1})\\
p_{\alpha}(x_{2})+\pi_{1}\rho_{2}\\
\vdots\\
p_{\alpha}(x_{D-1})+\pi_{1}\rho_{D-1}
\end{array}\right).
\]
The next lemma is about the rank of $V(x_{1})$.
\begin{lem}
If $x_{1},\ldots,x_{D-1}$ are distinct, the rank of $V(x_{1})$ is
equal to the number of its rows.\label{lem:10-V(x1)-rank}\end{lem}
\begin{proof}
Suppose we consider 
\[
\mathcal{V}=\left(\begin{array}{c}
p_{\alpha}(x_{1})\\
p_{\alpha}(x_{2})\\
\vdots\\
p_{\alpha}(x_{D-1})
\end{array}\right).
\]
The rank lemma \ref{lem:3-vandermonde} tells us that the rank of
$\mathcal{V}$ is equal to the number of its rows. 

Now to produce a vector $(\mathfrak{a}_{1},\ldots,\mathfrak{a}_{D-1})^{T}$
in the range of $V(x_{1})$, we proceed as follows. Define 
\begin{align*}
\mathfrak{a}_{1}' & =\mathfrak{a}_{1}\\
\mathfrak{a}_{2}' & =\mathfrak{a}_{2}-\rho(x_{1},\mathfrak{a}_{1}')\\
\mathfrak{a}_{3}' & =\mathfrak{a}_{3}-\rho(x_{1},x_{2},\mathfrak{a}_{1}',\mathfrak{a}_{2}')
\end{align*}
and so on. Because of the linearity of $\rho$ in $\mathfrak{a}_{i}$,
the vector $(c_{\alpha})$ that satisfies $\mathcal{V}(c_{\alpha})=(\mathfrak{a}_{1}',\ldots,\mathfrak{a}_{D-1}')^{T}$
also satisfies $V(x_{1})(c_{\alpha})=(\mathfrak{a}_{1},\ldots,\mathfrak{a}_{D-1})^{T}.$
\end{proof}
The next lemma is similar. Part (c) of the following lemma is more
general than Lemma \ref{lem:10-V(x1)-rank} because we allow $D^{+}>D$.
\begin{lem}
The following matrices have rank equal to the number of rows: \label{lem:11-V(x1)-V(y1)-rank}\end{lem}
\begin{enumerate}[label=\alph{enumi}]
\item  $\left(\begin{array}{c}
V(x_{1})\\
V(y_{1})
\end{array}\right)$ assuming $x_{1},\ldots,x_{D-1},y_{1},\ldots,y_{D-1}$ to be distinct.
\item $\left(\begin{array}{c}
V(x_{1})\\
\mathfrak{m}_{k}
\end{array}\right),$ where $\mathfrak{m}_{k}$ is the first $k$ rows of $V(y_{1})$,
assuming $x_{1},\ldots,x_{D-1},y_{1},\ldots,y_{k}$ to be distinct. 
\item $\left(\begin{array}{c}
p_{\alpha}(x_{1})\\
p_{\alpha}(x_{2})+\pi_{1}\rho_{2}\\
\vdots\\
p_{\alpha}(x_{D^{+}-1})+\pi_{1}\rho_{D^{+}-1}
\end{array}\right)$ assuming $x_{1},\ldots,x_{D^{+}-1}$ are distinct and $D^{+}\leq2D$.\end{enumerate}
\begin{proof}
Similar to the previous proof.
\end{proof}
Our second task in this section is to obtain $dF_{\alpha}(x_{1},v_{1})$
as a perturbation of $dF_{0}(x_{1},v_{1})$. It is helpful to introduce
another convention:
\begin{description}
\item [{Convention}] about $w$: $w_{1}=v_{1}$, $w_{2}$ is obtained as
$\frac{\partial\phi_{\alpha}}{\partial x}\Bigl|_{\tilde{x}_{1}}w_{1}$,
$w_{3}$ is obtained as $\frac{\partial\phi_{\alpha}}{\partial x}\Bigl|_{\tilde{x}_{2}}w_{2}$,
and so on.
\end{description}
Thus, in effect we need to obtain perturbative expansions of $w_{i}$.
To do so, let us first note that 
\[
\frac{\partial\phi_{\alpha}}{\partial x}=\psi(x)+\mathbf{e}_{1}\left(\nabla p_{\alpha}(x)^{T}\right)(c_{\alpha}).
\]
We substitute the above equation into the iteration that defines $w_{i}$
and obtain
\begin{align*}
w_{1} & =v_{1}\\
w_{2} & =\psi(\tilde{x}_{1})w_{1}+\mathbf{e}_{1}\left(v_{1}^{T}\nabla p_{\alpha}(x_{1})\right)(c_{\alpha})\\
w_{3} & =\psi(\tilde{x}_{2})w_{2}+\mathbf{e}_{1}\left(v_{2}^{T}\nabla p_{\alpha}(x_{2})\right)(c_{\alpha})+\varrho_{2}+\mathcal{O}(c_{\alpha}^{2})
\end{align*}
and so on. If we now use (\ref{eq:rho-definition}) to substitute
for $\tilde{x}_{j}$, we obtain 
\[
w_{j}=v_{j}+\mathbf{e}_{1}\left(v_{j-1}^{T}\nabla p_{\alpha}(x_{j-1})\right)(c_{\alpha})+\varrho_{j-1}(c_{\alpha})+\mathcal{O}(c_{\alpha}^{2}),
\]
where $\rho_{j-1}$ is linear in 
\[
p_{\alpha}(x_{1}),\ldots,p_{\alpha}(x_{j-2}),\nabla p_{\alpha}(x_{1}),\ldots,\nabla p_{\alpha}(x_{j-2}).
\]
We may then write
\begin{align*}
dF_{\alpha}(x_{1},v_{1}) & =\left(\begin{array}{c}
\pi_{1}w_{1}\\
\vdots\\
\pi_{1}w_{D}
\end{array}\right)\\
 & =dF_{0}(x_{1},v_{1})+\left(\begin{array}{c}
0\\
H(x_{1},v_{1})
\end{array}\right)(c_{\alpha})+\mathcal{O}(c_{\alpha}^{2}),
\end{align*}
where 
\[
H(x_{1},v_{1})=\left(\begin{array}{c}
v_{1}^{T}\nabla p_{\alpha}(x_{1})\\
v_{2}^{T}\nabla p_{\alpha}(x_{2})+\pi_{1}\varrho_{2}\\
v_{3}^{T}\nabla p_{\alpha}(x_{3})+\pi_{1}\varrho_{3}\\
\vdots\\
v_{D-1}^{T}\nabla p_{\alpha}(x_{D-1})+\pi_{1}\varrho_{D-1}
\end{array}\right).
\]
The second task for this section concludes with a lemma about the
rank of $H(x_{1},v_{1})$.
\begin{lem}
If $x_{1},\ldots,x_{D-1}$ are distinct, the rank of $H(x_{1},v_{1})$
is equal to the number of its rows.\label{lem:12-H(x1,v1)-rank}\end{lem}
\begin{proof}
The proof is similar to that of Lemma \ref{lem:10-V(x1)-rank}. First
consider
\[
\left(\begin{array}{c}
p_{\alpha}(x_{1})\\
\vdots\\
p_{\alpha}(x_{D-1})\\
v_{1}^{T}\nabla p_{\alpha}(x_{1})\\
\vdots\\
v_{D-1}^{T}\nabla p_{\alpha}(x_{D-1})
\end{array}\right).
\]
By Lemma \ref{lem:5-complete-Hermite}, the rank of this matrix is
equal to the number of its rows. Suppose we want to find $(c_{\alpha})$
such that $H(x_{1},v_{1})(c_{\alpha})$ equals a specified vector
$(\mathfrak{a}_{1},\ldots,\mathfrak{a}_{D-1})^{T}$. To do so, we
find a vector $(c_{\alpha})$ such that the matrix displayed above
applied to $c_{\alpha}$ is equal to 
\[
\left(\begin{array}{c}
0\\
\vdots\\
0\\
\mathfrak{a}_{1}'\\
\vdots\\
\mathfrak{a}_{D-1}'
\end{array}\right),
\]
where $\mathfrak{a}_{1}'=\mathfrak{a}_{1}$, $\mathfrak{a}_{2}'=\mathfrak{a}_{2}-\mathfrak{r}_{2}$,
where $\mathfrak{r}_{2}$ is $\pi\varrho_{2}$ evaluated by replacing
$p_{\alpha}(x_{1})$ by $0$ and $v_{1}^{T}\nabla p_{\alpha}(x_{1})$
by $\mathfrak{a_{1}}',$ and so on.
\end{proof}
The third and final task of this section is to track the perturbation
of fixed points when the map $\phi$ is perturbed to $\phi_{\alpha}$.
\begin{lem}
Suppose $z_{0}=\phi(z_{0})$ and $\psi(z_{0})$ has no eigenvalue
equal to $1.$ Under $\phi\rightarrow\phi_{\alpha}$, the fixed point
$z_{0}$ perturbs to 
\[
z_{0}(c_{\alpha})=z_{0}+(I-\psi(z_{0}))^{-1}\mathbf{e}_{1}\left(p_{\alpha}(z_{0})\right)(c_{\alpha})+\mathcal{O}(c_{\alpha}^{2}).
\]
\label{lem:13-fixed-point-tracking}\end{lem}
\begin{proof}
The function $z_{0}(c_{\alpha})$ exists by the implicit function
theorem. To obtain the expansion given in the lemma, start with 
\[
\phi(z_{0})+\mathbf{e}_{1}\left(p_{\alpha}(z_{0})\right)(c_{\alpha})=z_{0}
\]
differentiate with respect to $c_{\alpha}$ and obtain $\frac{\partial z_{0}}{\partial c_{\alpha}}$
at $c_{\alpha}=0$ using implicit differentiation.
\end{proof}

\section{The setting for injectivity and immersivity theorems}

In the case where $\phi$ is fixed and only the observation function
$o$ is perturbed, injectivity and immersivity are proved with respect
to the ball $\norm{c_{\alpha}}\leq a_{0}$, where $a_{0}>0$ can be
anything. Such a thing is plainly impossibly when $\phi$ is perturbed
to $\phi_{\alpha}$. Under a perturbation, the map may even fail to
be well defined or might blow-up in finite time. Therefore, we have
to specify the setting for injectivity and immersivity theorems more
carefully.

We will assume that $K$ is a compact sphere in $\mathbb{R}^{d}$
centered at the origin. The map $\phi_{\alpha}$ will be proved to
be injective and immersive over $K$. It is assumed that $K^{+}$
is a compact sphere bigger than $K$ and containing $K$. If $x_{1}\in K$,
it is assumed that $x_{1},\ldots,x_{D}$ all remain in $K^{+}$. In
addition, $a_{0}$ is assumed to be so small that $\tilde{x}_{1},\ldots,\tilde{x}_{D}$
all remain in $K^{+}$ for all $\norm{c_{\alpha}}\leq a_{0}$. Further
assumptions are enumerated below:
\begin{enumerate}
\item $\phi_{\alpha}:\mathbb{R}^{d}\rightarrow\mathbb{R}^{d}$ is assumed
to be a diffeomorphism (for $\norm{c_{\alpha}}\leq a_{0}$), that
is $C^{3}$ or better.
\item The map $\phi_{\alpha}$ has exactly $m$ fixed points and those will
be denoted by $\xi_{1}(c_{\alpha}),\ldots,\xi_{m}(c_{\alpha})$.
\item The map $\phi_{\alpha}$ has no other periodic points of period less
than $2D$.
\item All the fixed points are hyperbolic and $\pi_{1}\xi_{i}(c_{\alpha})\neq\pi_{1}\xi_{j}(c_{\alpha})$
if $i\neq j$. This assumption is made with the intention of simplifying
the proof so as to bring out the main techniques with greater clarity.
Here we are essentially assuming injectivity between fixed points.
\item We will also assume that $dF_{\alpha}$ is immersive at each fixed
point for the same reason.
\end{enumerate}
Now we will recall a few basic facts about Lebesgue points. A point
$\mathfrak{a\in}\mathbb{R}^{n}$ is a Lebesgue point of a measurable
set $A\subset\mathbb{R}^{n}$ if 
\[
\lim_{\epsilon\rightarrow0}\frac{\mu\left(A\cap\set{u\bigl|\norm{u-\mathfrak{a}}<\epsilon}\right)}{\mu\left(\set{u\bigl|\norm{u-\mathfrak{a}}<\epsilon}\right)}=1.
\]
We will need the following basic lemma.
\begin{lem}
If every point of the measurable set $B$ is a Lebesgue point of the
measurable set $A$, then $\mu(B-A)=0$.\label{lem:14-Lebesgue-points}\end{lem}
\begin{proof}
Almost every point of $A$ is a Lebesgue point of $A$ \cite{Rudin1986}.
Similarly, almost every point of $A^{c}$, the complement of $A$,
is a Lebesgue point of $A^{c}$. If $\mathfrak{a}$ is a Lebesgue
point of $A^{(c)}$, 

\[
\lim_{\epsilon\rightarrow0}\frac{\mu\left(A\cap\set{u\bigl|\norm{u-\mathfrak{a}}<\epsilon}\right)}{\mu\left(\set{u\bigl|\norm{u-\mathfrak{a}}<\epsilon}\right)}=0.
\]
The lemma follows from these observations.
\end{proof}
Lemma \ref{lem:14-Lebesgue-points} will be crucial to our proof that
$\phi_{\alpha}$ is an embedding with probability $1$ relative to
$\norm{c_{\alpha}}<a_{0}$. In the case where $\phi$ is fixed and
only the observation function is perturbed, the proofs of injectivity
and immersivity consider the ball $\norm{c_{\alpha}}\leq a_{0}$ all
at once. Such a thing is not possible here. Instead, we have to pick
$c_{\alpha}^{*}$ satisfying $\norm{c_{\alpha}^{\ast}}<a_{0}$ and
localize around it and that is where Lemma \ref{lem:14-Lebesgue-points}
comes in. 

In order to localize around $c_{\alpha}^{\ast}$, we adopt new notation
that is centered at $c_{\alpha}^{\ast}$. The re-centered diffeomorphism
$\phi(x)+\mathbf{e}_{1}(p_{\alpha}(x))(c_{\alpha}^{\ast})$ is denoted
by $\Phi(x)$. Similarly, $\Psi$ denotes $\psi(x)+\mathbf{e}_{1}(\nabla p_{\alpha}(x))(c_{\alpha}^{\ast})$.
When we localize around $c_{\alpha}^{\ast}$, $\Phi_{\alpha}(x)$
will denote $\Phi(x)+\mathbf{e}_{1}(p_{\alpha}(x))(c_{\alpha})$.
The fixed point $\xi_{j}(c_{\alpha}^{\ast})$ is denoted $\Sigma_{j}$.
The fixed point $\xi_{j}(c_{\alpha}^{\ast}+c_{\alpha})$ is denoted
$\Sigma_{j}(c_{\alpha})$.
\begin{description}
\item [{Convention}] about $x,y$ updated: $x_{1},x_{2},\ldots$ are iterates
of $x_{1}$ under $\Phi$. Similarly, $y_{1},y_{2},\ldots$ are iterates
of $y_{1}$ under $\Phi$.
\item [{Convention}] about $\tilde{x}$ updated: $\tilde{x}_{1}=x_{1}$
and $\tilde{x}_{1},\tilde{x}_{2},\ldots$ are iterates of $x_{1}$
under $\Phi_{\alpha}$.
\item [{Convention}] about $v$ updated: we assume $(x_{1},v_{1})\in T_{1}K$
and $v_{2},v_{3},\ldots$ are obtained by iterating $d\Phi$.
\end{description}
All the lemmas of the previous section continue to hold after re-centering.
The delay vector $F_{\alpha}(x)$ defined in the previous section
will be denoted by $\mathbb{F}_{0}(x)$ if $c_{\alpha}$ is replaced
by $c_{\alpha}^{\ast}$. Similarly, if $c_{\alpha}$ is replaced by
$c_{\alpha}^{\ast}+c_{\alpha}$ in the definition of $F_{\alpha}(x)$,
we will denote the re-centered delay vector by $\mathbb{F}_{\alpha}(x)$.

We may write
\[
\mathbb{F}_{\alpha}(x_{1})=\mathbb{F}_{0}(x_{1})+\left(\begin{array}{c}
0\\
\mathbb{V}(x_{1})
\end{array}\right)(c_{\alpha})+\mathcal{O}(c_{\alpha}^{2}),
\]
with the definition of $\mathbb{V}(x_{1})$ being the same as that
of $V(x_{1})$ but with $\psi$ replaced by $\Psi$. Likewise,
\[
d\mathbb{F}_{\alpha}(x_{1},v_{1})=d\mathbb{F}_{0}(x_{1},v_{1})+\left(\begin{array}{c}
0\\
\mathbb{H}(x_{1},v_{1})
\end{array}\right)(c_{\alpha})+\mathcal{O}(c_{\alpha}^{2}),
\]
with a similar alteration of the definition of $H(x_{1},v_{1})$ to
get $\mathbb{H}(x_{1},v_{1})$.

Finally, we note that the centered analogue of $G_{\alpha}(x_{1},y_{1})=F_{\alpha}(x_{1})-F_{\alpha}(y_{1})$
is $\mathbb{G}_{\alpha}(x_{1},y_{1})=\mathbb{F}_{\alpha}(x_{1})-\mathbb{F}_{\alpha}(y_{1})$.

\section{Proof of injectivity}

In this section, our purpose is to prove that $F_{\alpha}(x_{1})$,
defined in section 5, is injective for $x_{1}\in K$. The assumptions
about $C_{K}$ and $L$ are carried forward from earlier sections,
although the third assumption about $C_{K}$ is not necessary in its
entirety. Further assumptions will be stated as the need arises. Let
us define $\Delta$ is the minimum distance between fixed points of
$F_{\alpha}$ in $K$ for $\norm{c_{\alpha}}\leq a_{0}$.

Let us define $\mathcal{A}_{1,\delta}$ to be the set of $c_{\alpha}$
satisfying 
\begin{enumerate}[nolistsep]
\item $\norm{c_{\alpha}}<a_{0}$
\item $G_{\alpha}(\xi_{j}(\alpha),x_{1})\neq0$ for $j\in\{1,\ldots,m\}$
and $x_{1}\in K$ with $\norm{x_{1}-\xi_{j}(\alpha}\geq3\delta$ for
each $j\in\{1,\ldots,m\}$.
\end{enumerate}
In this section and the next, we always assume $\delta<\Delta/3$.
\begin{lem}
If $D\geq2d+2$, every point of $\norm{c_{\alpha}}<a_{0}$ is a Lebesgue
point of $\mathcal{A}_{1,\delta}$ and therefore the probability of
$\mathcal{A}_{1,\delta}$ relative to the open ball $\norm{c_{\alpha}}<1$
is $1$.\label{lem:15-injectivity-A1}\end{lem}
\begin{proof}
Pick $c_{\alpha}^{\ast}$ satisfying $\norm{c_{\alpha}^{\ast}}<a_{0}$.
We will use an argument centered at $c_{\alpha}^{\ast}$ to show that
$c_{\alpha}^{\ast}$ is a Lebesgue point of $\mathcal{A}_{1,\delta}$.

Pick $a_{1}>0$ so small that $\norm{\Sigma_{j}(\alpha)-\Sigma_{j}}<\delta$
for $\norm{c_{\alpha}}\leq a_{1}$. Define $\mathcal{K}_{1,\delta}$
as the set of $x_{1}\in K$ such that $\norm{x_{1}-\Sigma_{j}}\geq2\delta$
for each $j\in\{1,\ldots,m\}$.

Let us look at $\mathbb{G}_{\alpha}(\Sigma_{j}(c_{\alpha}),x_{1})$.
Using Lemma \ref{lem:13-fixed-point-tracking} and the definition
of $\mathbb{V}(x_{1})$, we get 
\begin{equation}
\mathbb{G}_{\alpha}(\Sigma_{j}(c_{\alpha}),x_{1})=\mathbb{G}_{0}(\Sigma_{j},x_{1})+\mathcal{M}(c_{\alpha})+\mathcal{O}(c_{\alpha}^{2})\label{eq:lem15-1}
\end{equation}
with $\mathcal{M}=J\mathcal{V}$ and 
\[
J=\left(\begin{array}{ccccc}
 &  &  &  & 1\\
-1 &  &  &  & 1\\
 & -1 &  &  & 1\\
 &  & \ddots &  & \vdots\\
 &  &  & -1 & 1
\end{array}\right),\:\:\mathcal{V}=\left(\begin{array}{c}
\mathbb{V}(x_{1})\\
\pi_{1}(I-\Psi(\Sigma_{j}))^{-1}\mathbf{e}_{1}p_{\alpha}(\Sigma_{j})
\end{array}\right).
\]
There are two cases here. Suppose $\pi_{1}(I-\Psi(\Sigma_{j}))^{-1}\mathbf{e}_{1}$
is nonzero. Then by Lemma \ref{lem:11-V(x1)-V(y1)-rank} (b), the
rank of $\mathcal{V}$ is equal to the number of its rows. Therefore,
the rank $J\mathcal{V}$ is $D$. If in fact the corner entry $\pi_{1}(I-\Psi(\Sigma_{j}))^{-1}\mathbf{e}_{1}$
is zero, we can drop the last column and first row of $J$ and conclude
that the rank of $J\mathcal{V}$ is $D-1$. In either case, the rank
of $\mathcal{M}$ is $D-1$ or greater.

Define $\sigma_{\delta}=\min\sigma_{D-1}(\mathcal{M})$, where the
minimum is over $x_{1}\in\mathcal{K}_{1,\delta}$ and $\norm{c_{\alpha}}\leq a_{1}$.
Cover $\mathcal{K}_{1}(\delta)$ with $C_{K}/\epsilon^{d}$ $\epsilon$-balls. 
\begin{description}
\item [{Assumption}] about $L$ (3): In (\ref{eq:lem15-1}), the $\mathcal{O}(c_{\alpha}^{2})$
term is upper bounded by $L\norm{c_{\alpha}}^{2}$. Like the earlier
assumptions about $L$, this assumption too is a direct consequence
of compactness. The earlier assumptions used $L$ as a bound on Lipshitz
constants. Here $L$ is used as a bound on the Taylor series remainder.
\end{description}

Now suppose $\mathbb{G}_{\alpha}(\Sigma_{j}(c_{\alpha}),x_{1})=0$
for some $j\in\set{1,\ldots,m}$ and some $x_{1}\in\mathcal{K}_{1,\delta}$.
Because the Lipshitz constant of $\mathbb{G}_{\alpha}(\Sigma_{j}(c_{\alpha}),x_{1})$
with respect to $x_{1}$ is bounded by $L$, we must have $\norm{\mathbb{G}_{\alpha}(\Sigma_{j}(c_{\alpha}),x_{1})}\leq L\epsilon$
at an $x_{1}$ that is at the center of one the balls covering $\mathcal{K}_{1,\delta}$. 

Applying the nonlinear transfer of volume Lemma \ref{lem:2-transfer-of-volume-nonlinear}
with $\mathfrak{r}\leftarrow D-1$ and $\sigma\leftarrow\sigma_{\delta}$,
we find that the probability of $\norm{\mathbb{G}_{\alpha}(\Sigma_{j}(c_{\alpha}),x_{1})}\leq L\epsilon$
relative to $\norm{c_{\alpha}}\leq\epsilon^{1/2}<a_{1}$ is upper
bounded by 
\[
D_{\alpha}!2^{D-1}L^{D-1}\epsilon^{(D-1)/2}\Bigl/\sigma_{\delta}^{D-1}.
\]
Because the number of fixed points is $m$ and the number balls covering
$\mathcal{K}_{1,\delta}$ is $C_{K}/\epsilon^{d}$, the probability
of $\mathbb{G}_{\alpha}(\Sigma_{j}(c_{\alpha}),x_{1})=0$ for some
$j\in\set{1,\ldots,m}$ and some $x_{1}\in\mathcal{K}_{1,\delta}$
relative to $\norm{c_{\alpha}}\leq\epsilon^{1/2}$ is upper bounded
by 
\[
m\times\frac{C_{K}}{\epsilon^{d}}\times\frac{D_{\alpha}!2^{D-1}L^{D-1}\epsilon^{(D-1)/2}}{\sigma_{\delta}^{D-1}}.
\]
Evidently, the probability goes to zero as $\epsilon\rightarrow0$
if $D\geq2d+2$. Thus, we have shown that $c_{\alpha}^{\ast}$ is
a Lebesgue point of $\mathcal{A}_{1,\delta}$ proving the lemma.
\end{proof}
Now define $\mathcal{A}_{2,\delta}$ to be the set of $c_{\alpha}$
satisfying 
\begin{enumerate}[nolistsep]
\item $\norm{c_{\alpha}}<a_{0}$
\item $G_{\alpha}(x_{1},\phi_{\alpha}(x_{1}))\neq0$ for $x_{1}\in K$ with
$\norm{x_{1}-\xi_{j}(\alpha}\geq3\delta$ for each $j\in\{1,\ldots,m\}$.\end{enumerate}
\begin{lem}
If $D\geq2d+1$, every point of $\norm{c_{\alpha}}<a_{0}$ is a Lebesgue
point of $\mathcal{A}_{2,\delta}$ and therefore the probability of
$\mathcal{A}_{2,\delta}$ relative to $\norm{c_{\alpha}}<a_{0}$ is
$1$.\label{lem:16-injectivity-A2}\end{lem}
\begin{proof}
As before, we pick $c_{\alpha}^{\ast}$ satisfying $\norm{c_{\alpha}^{\ast}}<a_{0}$
and will give an argument centered at $c_{\alpha}^{\ast}$ to show
that $c_{\alpha}^{\ast}$ is a Lebesgue point of $\mathcal{A}_{1,\delta}$.
As before, pick $a_{1}>0$ so small that $\norm{\Sigma_{j}(\alpha)-\Sigma_{j}}<\delta$
for $\norm{c_{\alpha}}\leq a_{1}$. As before, define $\mathcal{K}_{1,\delta}$
as the set of $x_{1}\in K$ such that $\norm{x_{1}-\Sigma_{j}}\geq2\delta$
for each $j\in\{1,\ldots,m\}$.

Using (\ref{eq:rho-defn-brief}), we get 
\begin{equation}
\mathbb{G}_{\alpha}(\tilde{x}_{1},\tilde{x}_{2})=\left(\begin{array}{c}
\pi_{1}x_{1}-\pi_{1}x_{2}\\
\vdots\\
\pi_{1}x_{D}-\pi_{1}x_{D+1}
\end{array}\right)+\mathcal{M}(c_{\alpha})+\mathcal{O}(c_{\alpha}^{2})\label{eq:lem16-1}
\end{equation}
with $\mathcal{M}=J\mathcal{V}$ and 
\[
J=\left(\begin{array}{cccc}
-1\\
1 & -1\\
 & 1 & -1\\
 &  & 1 & -1
\end{array}\right),\:\:\mathcal{V}=\left(\begin{array}{c}
p_{\alpha}(x_{1})\\
p_{\alpha}(x_{2})+\pi_{1}\rho_{2}\\
\vdots\\
p_{\alpha}(x_{D})+\pi_{1}\rho_{D}
\end{array}\right).
\]
By Lemma \ref{lem:11-V(x1)-V(y1)-rank} (c), the rank of $\mathcal{V}$
is equal to the number of its rows. Therefore, the rank of $\mathcal{M}=J\mathcal{V}$
is equal to $D$.

Define $\sigma_{\delta}=\min\sigma_{D}(\mathcal{M})$, where the minimum
is over $x_{1}\in\mathcal{K}_{1,\delta}$ and $\norm{c_{\alpha}}\leq a_{1}$.
Cover $\mathcal{K}_{1}(\delta)$ with $C_{K}/\epsilon^{d}$ $\epsilon$-balls.
\begin{description}
\item [{Assumption}] about $L$ (4): In (\ref{eq:lem16-1}), the $\mathcal{O}(c_{\alpha}^{2})$
term is upper bounded by $L\norm{c_{\alpha}^{2}}$. The first two
assumptions about $L$ are both obtained from upper bounds on the
derivative of $F_{\alpha}(x)$ or $\mathbb{F}_{\alpha}(x)$ with respect
to $x$. This assumption as well as the preceding one are obtained
from upper bounds on the second derivative. In all cases, the assumptions
are direct consequences of the compactness of $K$ and the ball $\norm{c_{\alpha}}\leq a_{0}$.
\end{description}

If $\mathbb{G}_{\alpha}(\tilde{x}_{1},\tilde{x}_{2})=0$ for some
$x_{1}\in\mathcal{K}_{1,\delta}$, we must have $\norm{\mathbb{G}_{\alpha}(\tilde{x}_{1},\tilde{x}_{2})}\leq L\epsilon$
for some $x_{1}$ that is the center of one of the balls covering
$\mathcal{K}_{1,\delta}$. Using the nonlinear transfer of volume
Lemma \ref{lem:2-transfer-of-volume-nonlinear}, we find the probability
of $\mathbb{G}_{\alpha}(\tilde{x}_{1},\tilde{x}_{2})=0$ for some
$x_{1}\in\mathcal{K}_{1,\delta}$ relative to the ball $\norm{c_{\alpha}}\leq\epsilon^{1/2}<a_{1}$
to be upper bounded by 
\[
\frac{C_{K}}{\epsilon^{d}}\times\frac{D_{\alpha}!2^{D}L^{D}\epsilon^{D/2}}{\sigma_{\delta}^{D}}.
\]
The limit of this probability as $\epsilon\rightarrow0$ is zero.
It follows that $c_{\alpha}^{\ast}$ is a Lebesgue point of $\mathcal{A}_{2,\delta}$
completing the proof of this lemma.
\end{proof}
Lemma \ref{lem:16-injectivity-A2} allows us to conclude that the
delay vectors of $x_{1}$ and $\phi_{\alpha}(x_{1})$ do not coincide
typically if $x_{1}$is a little removed from the fixed points of
$\phi_{\alpha}$. More generally, we need to argue that the delay
vectors of $x_{1}$ and $\phi_{\alpha}^{k-1}(x)$ do not coincide
for $k=3,\ldots,D$. To make that argument, we define $\mathcal{A}_{k,\delta}$
to be the set of $c_{\alpha}$ satisfying
\begin{enumerate}[nolistsep]
\item $\norm{c_{\alpha}}<a_{0}$
\item $G_{\alpha}(x_{1},\phi_{\alpha}^{k-1}(x_{1}))\neq0$ for $x_{1}\in K$
with $\norm{x_{1}-\xi_{j}(\alpha}\geq3\delta$ for each $j\in\{1,\ldots,m\}$
\end{enumerate}
for $k=2,\ldots,D$. 
\begin{lem}
For $D\geq2d+1$ and $k=2,\ldots,D$, every point of $\norm{c_{\alpha}}<a_{0}$
is a Lebesgue point of $\mathcal{A}_{k,\delta}$ and therefore the
probability of $\mathcal{A}_{k,\delta}$ relative to the ball $\norm{c_{\alpha}}<a_{0}$
is $1$.\label{lem:17-injectivity-Ak}\end{lem}
\begin{proof}
The proof is almost identical to that of the previous lemma, which
is a special case. The only significant difference occurs in the definition
of $\mathcal{V}$. In the general case, 
\[
\mathcal{V}=\left(\begin{array}{c}
p_{\alpha}(x_{1})\\
p_{\alpha}(x_{2})+\pi_{1}\rho_{2}\\
\vdots\\
p_{\alpha}(x_{D+k-2})+\pi_{1}\rho_{D+k-2}
\end{array}\right).
\]
Note that Lemma \ref{lem:11-V(x1)-V(y1)-rank} (c) still applies,
implying the rank of $\mathcal{V}$ to be equal to the number of its
rows, because $D+k-2\leq2D$.
\end{proof}
The final lemma of this section pertains to the set $\mathcal{A}_{xy}(\delta)$.
It is defined as the set of all $c_{\alpha}$ such that $\norm{c_{\alpha}}<a_{0}$
and $G_{\alpha}(x_{1},y_{1})\neq0$ provided 
\begin{enumerate}[nolistsep]
\item $x_{1},y_{1}\in K$
\item $\norm{x_{1}-y_{1}}\geq\delta$ (which excludes the diagonal of $K\times K$)
\item $\norm{x_{1}-\xi_{j}(\alpha)}\geq3\delta$ and $\norm{y_{1}-\xi_{j}(\alpha)}\geq3\delta$
for $j\in\{1,\ldots,m\}$(so that both $x_{1}$ and $y_{1}$ stay
away from fixed points)
\item $\norm{x_{1}-\phi^{k-1}(y_{1})}\geq2\delta$ and $\norm{y_{1}-\phi^{k-1}(x_{1})}\geq2\delta$
for $k=2,\ldots,D$ (so that $x_{1}$ does not come too close to the
iterates of $y_{1}$ and vice versa).\end{enumerate}
\begin{lem}
For $D\geq4d+2$, every point of $\norm{c_{\alpha}}<a_{0}$ is a Lebesgue
point of $\mathcal{A}_{xy,\delta}$ and therefore the probability
of $\mathcal{A}_{xy,\delta}$ relative to the ball $\norm{c_{\alpha}}<a_{0}$
is $1$.\label{lem:18-injectivity-xy}\end{lem}
\begin{proof}
Again the argument begins by centering at some $c_{\alpha}^{\ast}$
satisfying $\norm{c_{\alpha}^{\ast}}<a_{0}$. However, the conditions
on $a_{1}$ this time are different. The radius $a_{1}$ must be so
small that for $\norm{c_{\alpha}}\leq a_{1}$ the following conditions
are satisfied:
\begin{enumerate}[nolistsep]
\item  $\norm{\Sigma_{j}(\alpha)-\Sigma_{j}}<\delta$
\item For any $x_{1}\in K$, $\norm{\tilde{x}_{j}-x_{j}}\leq\delta$ for
$j=1,\ldots,D$.
\end{enumerate}
The set $\mathcal{K}_{xy,\delta}$ is defined as the set of $(x_{1},y_{1})\in K\times K$
satisfying the following conditions:
\begin{enumerate}[nolistsep]
\item $\norm{x_{1}-\Sigma_{j}}\geq2\delta$ and $\norm{y_{1}-\Sigma_{j}}\geq2\delta$
for $j\in\{1,\ldots,m\}$
\item $\norm{x_{1}-y_{1}}\geq\delta$
\item $\norm{x_{1}-y_{j}}\geq\delta$ and $\norm{y_{1}-x_{j}}\geq\delta$
for $j\in\{2,\ldots,m\}$.
\end{enumerate}

We have 
\begin{equation}
\mathbb{G}_{\alpha}(x_{1},y_{1})=\mathbb{G}_{0}(x_{1},y_{1})+\mathcal{M}(c_{\alpha})+\mathcal{O}(c_{\alpha}^{2}).\label{eq:lem18-1}
\end{equation}
The top row of $\mathcal{M}$ is zero. The rest of the $D-1$ rows
below are given by $J\mathcal{V}$ 
\[
J=\left(\begin{array}{cccccc}
1 &  &  & -1\\
 & \ddots\\
 &  & 1 &  &  & -1
\end{array}\right),\:\:\mathcal{V}=\left(\begin{array}{c}
\mathbb{V}(x_{1})\\
\mathbb{V}(y_{1})
\end{array}\right).
\]
By Lemma \ref{lem:11-V(x1)-V(y1)-rank}, the rank of $\mathcal{V}$
is equal to the number of its rows. Therefore the ranks of $J\mathcal{V}$
and $\mathcal{M}$ are both equal to $D-1$.

Define $\sigma_{\delta}=\min\sigma_{D-1}(\mathcal{M})$, where the
minimum is over $(x_{1},y_{1})\in\mathcal{K}_{xy,\delta}$ and $\norm{c_{\alpha}}\leq a_{1}$.
Cover $\mathcal{K}_{xy}$ with $C_{K}/\epsilon^{2d}$ balls. 
\begin{description}
\item [{Assumption}] about $L$ (5): The $\mathcal{O}(c_{\alpha}^{2})$
term in (\ref{eq:lem18-1}) is upper bounded by $L\norm{c_{\alpha}}^{2}$.
\end{description}

Suppose $\mathbb{G}_{\alpha}(x_{1},y_{1})=0$ for some $(x_{1},y_{1})\in\mathcal{K}_{xy,\delta}$.
Then we must have $\norm{\mathbb{G}_{\alpha}(x_{1},y_{1})}\leq L\epsilon$
for an $(x_{1},y_{1})$ that is at the center of one of the balls
covering $\mathcal{K}_{xy,\delta}$. Applying the nonlinear transfer
of volume Lemma \ref{lem:2-transfer-of-volume-nonlinear}, we find
the probability of $\mathbb{G}_{\alpha}(x_{1},y_{1})=0$ for some
$(x_{1},y_{1})\in\mathcal{K}_{xy,\delta}$ relative to the ball $\norm{c_{\alpha}}\leq\epsilon^{1/2}<a_{1}$
to be upper bounded by 
\[
\frac{C_{K}}{\epsilon^{2d}}\times\frac{D_{\alpha}!2^{D-1}L^{D-1}\epsilon^{\frac{D-1}{2}}}{\sigma_{\delta}^{D}}.
\]
If $D\geq4d+2$, the limit of this probability as $\epsilon\rightarrow0$
is $0$. Therefore, every $c_{\alpha}^{\ast}$ satisfying $\norm{c_{\alpha}^{\ast}}<a_{0}$
is a Lebesgue point of $\mathcal{A}_{xy,\delta}$, which completes
the proof of the lemma.
\end{proof}
We are now prepared to state and prove the main theorem of this section.
\begin{thm}
Assuming $a_{0}$ and $\phi_{\alpha}$ satisfy the conditions laid
down in section 6 and $D\geq4d+2$, the delay mapping $F_{\alpha}$
is injective on the set $K$ with probability one relative to the
ball $\norm{c_{\alpha}}<a_{0}$.\end{thm}
\begin{proof}
The proof follows from Lemmas \ref{lem:15-injectivity-A1}, \ref{lem:17-injectivity-Ak},
and \ref{lem:18-injectivity-xy} by taking the limit $\delta\rightarrow0$
through a countable sequence.
\end{proof}

\section{Proof of immersivity}

All the main techniques have been demonstrated in the proof of injectivity
of the delay mapping $F_{\alpha}$. The assumption in section 6 that
$dF_{\alpha}$ is immersive at all fixed points in $K$ simplifies
the proof of immersivity considerably.

Define $\mathcal{A}_{T,\delta}$ as the set of all $c_{\alpha}$ satisfying
$\norm{c_{\alpha}}<a_{0}$ and $F_{\alpha}$ is immersive at all $x_{1}\in K$
satisfying $\norm{x_{1}-\xi_{j}(\alpha)}\geq3\delta$ for $j\in\{1,\ldots,m\}$.
In other words, we are requiring $dF_{\alpha}(x_{1},v_{1})\neq0$
if $(x_{1},v_{1})\in T_{1}K$ and $x_{1}$ is removed from each periodic
point by at least $3\delta$.
\begin{lem}
For $D\geq4d$, every point of $\norm{c_{\alpha}}<a_{0}$ is a Lebesgue
point of $\mathcal{A}_{T,\delta}$ and therefore the probability of
$\mathcal{A}_{T,\delta}$ relative to $\norm{c_{\alpha}}<a_{0}$ is
$1$.\label{lem:20-immersivity}\end{lem}
\begin{proof}
We center at $c_{\alpha}^{\ast}$ satisfying $\norm{c_{\alpha}^{\ast}}<a_{0}$
as before. Again as before, we assume $a_{1}$ to be so small that
$\norm{\Sigma_{j}(c_{\alpha})-\Sigma_{j}}<\delta$ for $\norm{c_{\alpha}}\leq a_{1}$.

Define $\mathcal{K}_{T,\delta}$ to be the set of all $(x_{1},v_{1})\in T_{1}K$
satisfying $\norm{x_{1}-\Sigma_{j}}\geq2\delta$ for $j\in\{1,\ldots,m\}$.
Then 
\begin{equation}
d\mathbb{F}_{\alpha}(x_{1},v_{1})=d\mathbb{F}_{0}(x_{1},v_{1})+\mathcal{N}(c_{\alpha})+\mathcal{O}(c_{\alpha}^{2})\label{eq:lem20-1}
\end{equation}
with 
\[
\mathcal{N}=\left(\begin{array}{c}
0\\
\mathbb{H}(x_{1},v_{1})
\end{array}\right).
\]
By Lemma \ref{lem:12-H(x1,v1)-rank}, the rank of $\mathcal{N}$ is
$D-1$. 

Define $\sigma_{\delta}=\min\sigma_{D-1}(\mathcal{N})$, where the
minimum is taken over $(x_{1},v_{1})\in\mathcal{K}_{T,\delta}$ and
$\norm{c_{\alpha}}\leq a_{1}$. Cover $\mathcal{K}_{T,\delta}$ with
$C_{K}/\epsilon^{2d-1}$ $\epsilon$-balls. 
\begin{description}
\item [{Assumption}] about $L$ (5): In (\ref{eq:lem20-1}), the $\mathcal{O}(c_{\alpha}^{2})$
term is upper bounded by $L\norm{c_{\alpha}}^{2}$. Here, we are effectively
assuming a bound on the third derivative of $F_{\alpha}(x_{1})$ with
respect to $x_{1}$ over the compact sets $x_{1}\in T_{1}K$ and $\norm{c_{\alpha}}\leq a_{0}$.
\end{description}

If $d\mathbb{F}_{\alpha}(x_{1},v_{1})=0$ for some $(x_{1},v_{1})\in\mathcal{K}_{T,\delta}$,
then we must have $\norm{d\mathbb{F}_{\alpha}(x_{1},v_{1})}$ for
some $(x_{1},v_{1})$ that is at the center of one of the $\epsilon$-balls
covering $\mathcal{K}_{T,\delta}$. The nonlinear transfer of volume
lemma \ref{lem:2-transfer-of-volume-nonlinear} implies that the probability
of $d\mathbb{F}_{\alpha}(x_{1},v_{1})=0$ for some $(x_{1},v_{1})\in\mathcal{K}_{T,\delta}$
relative to $\norm{c_{\alpha}}\leq\epsilon^{1/2}<a_{1}$ is upper
bounded by 
\[
\frac{C_{K}}{\epsilon^{2d-1}}\times\frac{D_{\alpha}!2^{D-1}L^{D-1}\epsilon^{\frac{D-1}{2}}}{\sigma_{\delta}^{D-1}}.
\]
If $D\geq4d$, this probability goes to zero as $\epsilon\rightarrow0$.
Therefore, every $\norm{c_{\alpha}^{\ast}}<a_{0}$ is a Lebesgue point
of $\mathcal{A}_{T,\delta}$, proving the lemma.
\end{proof}
We are now prepared to state and prove the immersivity theorem. 
\begin{thm}
Suppose $a_{0}$ and $\phi_{\alpha}$ satisfy the assumptions laid
down in section 6 and suppose $D\geq2d$. The delay map $F_{\alpha}$
is then immersive at every point of $K$ with probability $1$ relative
to the ball $\norm{c_{\alpha}}<a_{0}$.\end{thm}
\begin{proof}
The proof follows by taking $\delta\rightarrow0$ through a countable
sequence in the previous Lemma \ref{lem:20-immersivity} and using
the assumption made in section 6 about immersivity at fixed points.
\end{proof}

\section{Discussion}

The delay map may be viewed in light of the Whitney embedding theorem
\cite{Hirsch1976}. However, it has some characteristics of its own.
One of these is the possibility that orbits of two distinct points
can overlap. There are other distinctive characteristics related to
periodic orbits and eigenvectors.

In this article, we showed how to prove that the delay map is an embedding
using the concept of Lebesgue points. For the delay map $F_{\alpha}(x)$
with $o=\pi_{1}$ to be an embedding with probability $1$ relative
to the ball $\norm{c_{\alpha}}<1$, we require the embedding dimension
to satisfy $D\geq4d+2$.

We conjecture that the delay mapping is an embedding for $D\geq2d+1$.
The more restrictive $4d+2$ requirement comes in when applying the
nonlinear transfer of volume lemma. The extra dimensions are used
to absorb the effect of the nonlinear term. Some evidence for this
conjecture may be found in our earlier work \cite{NavarreteViswanath17}.

In our opinion, it would be desirable to obtain prevalence versions
of classical theorems such as the Kupka-Smale theorem \cite{Robinson1998}.
The differential topology proofs rely heavily on the bump function
and genericity is weaker than almost sureness in probability. It is
hoped that the technique based on Lebesgue points introduced here
will be useful in that regard.

\bibliographystyle{plain}
\bibliography{references}

\end{document}